\documentclass[12pt,leqno]{amsart}
\usepackage{amsmath,epsfig,graphicx,color, float}
\usepackage{subcaption}
\usepackage{relsize}
\usepackage{comment}
\usepackage[colorlinks, linkcolor=blue,citecolor=blue]{hyperref}
\usepackage{pdfpages}
\usepackage{graphicx}
\usepackage{mwe}
\usepackage{longtable} 
\usepackage{algorithm}
\usepackage{algorithmicx}
\usepackage{algpseudocode}
\usepackage{natbib}

\usepackage[margin=1.3in]{geometry}

\numberwithin{table}{section}
\numberwithin{figure}{section}
\numberwithin{equation}{section}

\definecolor{darkblue}{rgb}{.2, 0.2,.8}
\definecolor{darkgreen}{rgb}{0,0.5,0.3}
\definecolor{darkred}{rgb}{.8, .1,.1}

\newcommand{\bfx}{\vect{x}}
\newcommand{\bfX}{\vect{X}}

\newcommand{\bfalp}{\vect{\alpha}}

\newcommand{\bftheta}{\vect{\theta}}

\newcommand{\bfT}{\mat{T}}
\newcommand{\bft}{\vect{t}}
\newcommand{\bfe}{\vect{e}}
\newcommand{\bfpi}{\vect{\pi}^\mathsf{T}}
\newcommand{\bfp}{\vect{\pi}}

\newcommand{\0}{\mat{0}}

\newcommand{\E}{\mathbb{E}}
\renewcommand{\P }{{\mathbb P}}
\newcommand{\ci}{\mathrel{\text{\scalebox{1.07}{$\perp\mkern-10mu\perp$}}}}

\newcommand{\ov}{\overline}

\newcommand{\eqd}{\stackrel{d}{=}}

\newtheorem{lemma}{Lemma}[section]
\newtheorem{theorem}[lemma]{Theorem}
\newtheorem{proposition}[lemma]{Proposition}
\newtheorem{definition}[lemma]{Definition}

\newtheorem{example}[lemma]{Example}

\newtheorem{condition}[lemma]{Condition}
\newtheorem{remark}{Remark}[section]

\usepackage{bm}
\newcommand{\vect}[1]{\pmb{#1}}
\newcommand{\mat}[1]{\boldsymbol{\bm #1}}

\DeclareMathOperator*{\argmax}{arg\,max}

\usepackage{marginnote}

\allowdisplaybreaks

\begin{document}
\bibliographystyle{apalike}
\title[Phase-type mixture-of-experts regression for loss severities]{Phase-type mixture-of-experts regression for loss severities}

\author[M. Bladt]{Martin Bladt}
\address{Faculty of Business and Economics,
University of Lausanne,
Quartier de Chambronne,
1015 Lausanne,
Switzerland}
\email{martin.bladt@unil.ch}

\author[J. Yslas]{Jorge Yslas}
\address{Institute of Mathematical Statistics and Actuarial Science,
University of Bern,
Alpeneggstrasse 22,
CH-3012 Bern,
Switzerland}
\email{jorge.yslas@stat.unibe.ch}

\begin{abstract}
The task of modeling claim severities is addressed when data is not consistent with the classical regression assumptions. This framework is common in several lines of business within insurance and reinsurance, where catastrophic losses or heterogeneous sub-populations result in data difficult to model. Their correct analysis is required for pricing insurance products, and some of the most prevalent recent specifications in this direction are mixture-of-experts models. This paper proposes a regression model that generalizes the latter approach to the phase-type distribution setting. More specifically, the concept of mixing is extended to the case where an entire Markov jump process is unobserved and where states can communicate with each other. The covariates then act on the initial probabilities of such underlying chain, which play the role of expert weights. The basic properties of such a model are computed in terms of matrix functionals, and denseness properties are derived, demonstrating their flexibility. An effective estimation procedure is proposed, based on the EM algorithm and multinomial logistic regression, and subsequently illustrated using simulated and real-world datasets. The increased flexibility of the proposed models does not come at a high computational cost, and the motivation and interpretation are equally transparent to simpler MoE models.

\end{abstract}
\maketitle

\section{Introduction}
The correct estimation of claim severities is a classical problem in actuarial science, and yet the task remains challenging and often only solvable by partially formal procedures. For instance, when dealing with data arising from reinsurance of natural catastrophes or from third-party liability insurance, very large claims are treated differently to the bulk of smaller -- or attritional -- claim sizes. Multimodality of the attritional claims can further exacerbate the problem. Such heterogeneity in the data is often still present after segmentation, possibly due to pooled and unlabelled sub-populations in the dataset. In addition, most common commercial software, which mostly uses generalized linear models (GLM), does not capture quantiles correctly. Consequently, risk managers and actuaries interested in understanding their fitted probabilistic models (and not only using them for prediction) keep returning to the drawing board to obtain more interpretable, flexible and effective statistical tools for their practice.

Several statistically coherent approaches have been proposed in recent years to overcome multimodality and heavy-tailedness. For instance, mixing Erlang distributions results in multimodal histograms, \cite{lee2010modeling} being the first to consider such model for insurance, and then later extended by \cite{tzougas2014optimal, miljkovic2016modeling} for more general mixtures. More recent approaches, such as \cite{fung2019class}, adopt a mixture-of-experts approach, which consists of regressing the component probabilities of a finite mixture model. Regarding the heavy-tailed component, the formal way of dealing with the attritional and large claims jointly has been using splicing -- also referred to as composite models -- which have a different tail and body distribution (see \cite{grun2019extending} for a comparison and a good literature review). Combining the two approaches is the state-of-the-art of probabilistic models for loss severity modeling, referred to as composite models. \cite{reynkens2017modelling} were the first to consider this global approach, and \cite{fung2021mixture} suggested a feature-selection variant.

The main idea of this paper is to use phase-type (PH) distributions to capture the specificities of heterogeneous data more intuitively and effectively than mixing. The underlying multi-state model is easy to motivate and understand for practitioners and is mathematically convenient for developing their estimation. More specifically, we propose to use PH distributions to describe claim severities and build our regression framework with PH building blocks. PH distributions are defined as the absorption time of a time-homogeneous Markov pure-jump process on a finite state space. In life insurance, such a framework is familiar and understood as the traversing of healthy, disabled, and dead states, with time corresponding to calendar time. In non-life insurance, the states can be regarded as unobserved steps in legal cases or reparations of a building, and time now corresponds to the incurred monetary loss.

Many distributions such as the Erlang, generalized Coxian, and finite mixtures between them are all PH distributions (see \cite{neuts75,neuts1981matrix} for the first systematic approaches, and \cite{Bladt2017} for a recent comprehensive treatment), and they are even known to be dense in weak convergence on the set of distributions of positive-valued risks (cf. \cite{asmussen2008applied}). Most applications of PH were initially in the field of applied probability, but their estimation became widely used after \cite{asmussen1996fitting} laid out the EM-algorithm for statistical fitting. To correct for non-exponential tail behavior, \cite{albrecher2019inhomogeneous, albrecher2020iphfit} defined and provided estimation approaches for transformed PH distributions, also known as inhomogeneous phase-type (IPH) laws. 

To incorporate rating factors into our model, we consider regressing the initial probabilities of the underlying stochastic process starting in a given state. This is in the same spirit as the mixture-of-experts approaches, cf. \cite{yuksel2012twenty} for a survey, which can roughly be described as machine learning methods where inhomogeneous data regions are divided into homogeneous ones, where simpler models can suffice for their description. A different approach to regression with PH distributions was considered in
 \cite{bladt2021semi,pricing2021}, where the proportional intensities (PI) model was proposed. The strength of PH regression models is that no threshold selection is required, and a tail behavior specification can be easily done by choosing an appropriate inhomogeneity function. Moreover, the interaction between the hidden states allows for complex density shapes, going beyond what simple mixing can account for, for a given number of experts. In essence, the latter property can have a parsimonious effect on the number of estimation parameters. However, the phase-type mixture-of-experts (PH-MoE) approach can obtain a wider range of variation for a fixed state-space size than the PI approach and can be faster to estimate for a small number of covariates.

The PH-MoE model is rather flexible, illustrated by two denseness results on: a) multinomial experiments with arbitrary distributions assigned to each outcome; b) more general regression models, subject to some technical conditions. The marginal and conditional distributions associated with the PH-MoE specification fall into the IPH class, for which many closed-form formulas exist, and their tail behavior is well understood. Furthermore, their estimation can be carried out using an ingenious decomposition of the fully observed likelihood into two components: one which may be maximized using a variant of the EM algorithm for PH distributions; and another component can be seen as a weighted one multinomial logistic regression problem.

The remainder of the paper is structured as follows. First, in Section \ref{sec:PH-MoE}, we provide a short reminder of IPH distributions, specify the main regression model, and derive its basic properties, along with the first denseness result. We then prove the denseness of PH-MoE on regression models in Section \ref{sec:dens} and provide an effective estimation technique based on the EM algorithm and weighted multinomial regression in Section \ref{sec:est}, along with a goodness of fit consideration.
In Section~\ref{sec:transforms}, we review some common choices of inhomogeneity functions for global fitting and introduce a new alternative to composite splicing models based on piecewise-continuous inhomogeneity functions.
Subsequently, we show in Section \ref{sec:examples} the practical feasibility of our approach on synthetic and real insurance data. Finally, Section \ref{sec:conclusion} concludes.

\section{Phase-type Mixture-of-Experts regression model}\label{sec:PH-MoE}

\subsection{Preliminaries}
Let $ ( J_t )_{t \geq 0}$ be a time-inhomogeneous Mar\-kov pure-jump process on the finite state space $\{1, \dots, p, p+1\}$, where states $1,\dots,p$ are transient and $p+1$ is absorbing. Then, the transition probabilities
\begin{align*}
p_{kl}(s,t)=\P(J_t=l|J_s=k)\,,\quad 0\le k,l\le p+1 \,,
\end{align*}
can be written in matrix form as 
$$\mat{P}(s,t)=\prod_{s}^{t}(\boldsymbol{I}+\boldsymbol{\Lambda}(u) d u):=\boldsymbol{I}+\sum_{i=1}^{\infty} \int_{s}^{t} \int_{s}^{u_{i}} \cdots \int_{s}^{u_{2}} \mathbf{\Lambda}\left(u_{1}\right) \cdots \mathbf{\Lambda}\left(u_{i}\right) d u_{1} \cdots \mathrm{d} u_{i}\,,$$
for $s<t,$ where $\mat{\Lambda}(t)$ is called the intensity matrix, that is, a matrix with negative diagonal elements and non-negative off-diagonal elements such that the rows sum to zero. If we further require that the matrices $\mat{\Lambda}(s)$ and $\mat{\Lambda}(t)$ commute for every $s<t$, this can be done by assuming the following structure of the intensity matrix 
\begin{align*}
	\mat{\Lambda}(t)=\lambda(t) \left( \begin{array}{cc}
		\bfT &  \bft \\
		\0 & 0
	\end{array} \right)\in\mathbb{R}^{(p+1)\times(p+1)}\,, \quad t\geq0\,,
\end{align*}
where $\bfT $ is a $p \times p$ sub-intensity matrix, $\bft$ is a $p$-dimensional column vector providing the exit rates to the absorbing state, $\0$ is a $p$-dimensional row vector of zeroes, and $\lambda(\cdot)$ is some known positive real function. Since the rows of the intensity matrix sum to zero, the relationship $\bft=- \bfT \, \bfe$ holds, where $\bfe$ denotes the $p$-dimensional column vector of ones.
In what follows, we always assume this structure of $\mat{\Lambda}(t)$ and that the function $\lambda(\cdot)>0$  satisfies for $y>0$
\begin{align}\label{lambda_restrictions}
(0,\infty)\ni\int_0^y \lambda(t)dt\stackrel{y\to\infty}{\to}\infty\,.
\quad\end{align}

For future reference, we write $\bfe_k$ for the $k$-th canonical basis vector in $\mathbb{R}^p$.  Concerning the sub-intensity matrix and the vector of exit rates, we introduce the following notation for their entries
$$\bfT=(t_{kl})_{k,l=1,\dots,p}\,,\quad \bft=(t_1,\dots,t_p)^{\mathsf{T}}\,.$$

We will make use of functions of matrices in the sequel. The standard unambiguous way of defining them is in terms of the Cauchy formula as follows. Let $h$ be any analytic function and $\mat{A}$ a square matrix. Then we define
\begin{align*}
	h( \mat{A})=\dfrac{1}{2 \pi i} \oint_{\Gamma}h(w) (w \mat{I} -\mat{A} )^{-1}dw \,,\quad \mat{A}\in \mathbb{R}^{p\times p} \,,
\end{align*}
with $\Gamma$ a simple path enclosing the eigenvalues of $\mat{A}$, and $\mat{I}$ is the identity matrix of the same dimension.

\begin{definition}
Let $\bfp_0$ be an initial distribution on $\{1,\dots,p\}$. Then, if $J_0\sim\bfp_0$, we say that
$$Y_0 = \inf \{ t >  0 : J_t = p+1 \}\,,$$ follows an inhomogeneous phase-type (IPH) distribution and we write \linebreak $Y_0 \sim \mbox{IPH}(\bfp_0,\bfT,\lambda)$.
\end{definition}


The density $f$ and distribution function $F$ of $Y_0 \sim  \mbox{IPH}(\bfp_0 , \bfT , \lambda )$ are explicit in terms of functions of matrices and given by
\begin{eqnarray*}
 f(y) &=& \lambda (y)\, \vect{\pi}_0\exp \left( \int_0^y \lambda (s)ds\ \mat{T} \right)\vect{t}\,,\quad y\ge0 \,, \label{eq:dens-IPH} \\
 F(y)&=&  1- \vect{\pi}_0\exp \left( \int_0^y \lambda (s)ds\ \mat{T} \right)\vect{e}\,,\quad y\ge0 \,.  \label{eq:cdf-IPH}
\end{eqnarray*}
 
Another attractive property of IPH distributions is that a random variable following this specification can be expressed as the transformation of a phase-type (PH) distributed random variable, that is, the homogenous case corresponding to $\lambda \equiv1 $.  More specifically, if
 $Y_0 \sim  \mbox{IPH}(\bfp_0 , \bfT , \lambda )$, then 
  \begin{equation}\label{gtrans}
Y_0 \eqd g(Z_0) \,,
\end{equation}
where $Z_0 \sim \mbox{PH}(\bfp_0 , \bfT )$ and $g$ is defined  through its inverse in terms of $\lambda$ by
\begin{equation*}
g^{-1}(y) = \int_0^y \lambda (s)ds \,,  \quad y\ge 0\, . \label{eq:transformation-g}
\end{equation*}
This representation is particularly useful to derive further properties of IPH distributions by exploiting the known PH machinery. For instance, the following explicit asymptotic behavior for the tails $\ov F = 1- F$ of IPH distributions can be deduced using this representation in conjunction with the corresponding asymptotic result for PH distribution:
\begin{align}\label{eq:IPHasymptotic}
	\ov F(y) \sim c [g^{-1}(y)]^{m -1} \exp({-\eta g^{-1}(y)}) \,, \quad y \to \infty \,,
\end{align}
where $c$ is a positive constant depending on $\bfp$ and $\bfT$, $-\eta$ is the largest real eigenvalue of $\bfT$, and $m$ is the size of the Jordan block associated with $\eta$. 

\subsection{The regression model}

Define the mapping $$\bfp: D\subset \mathbb{R}^d\to \Delta^{p-1}\,,$$
where
$ \Delta^{p-1}=\{(\pi_1,\dots,\pi_p)\in\mathbb{R}^{p}\mid\sum_{k}{\pi_k} = 1 \mbox{ and } \pi_k \ge 0 \mbox{ for all } k\}$ is the standard $(p-1)$-simplex. Thus, for any given $\vect{x}\in\mathbb{R}^d$, we may endow the process with the initial probabilities
$$  \P(J_0 = k)=\pi_{k}(\vect{x}):=(\bfp(\vect{x}))_k\,,\quad k = 1,\dots, p\,,$$
and  $\P(J_0 = p + 1) = 0$. 

As a particular consequence, the following random variable
\begin{align*}
	Y = \inf \{ t >  0 : J_t = p+1 \}\,,
\end{align*}
satisfies that
\begin{align*}
	Y \sim \mbox{IPH}(\bfp(\vect{x}),\bfT,\lambda)\quad \Leftrightarrow \quad J_0\sim\bfp(\vect{x})\,.
\end{align*}

\begin{definition}
Let $\bfX$ be a $d$-dimensional vector of covariates. Then we say that
 $$Y|\,\bfX \sim\mbox{IPH}(\bfp(\bfX),\bfT,\lambda)$$ is a phase-type mixture-of-experts (PH-MoE) model.
\end{definition}

\begin{remark}\rm
The above model obtains its name since we may write
\begin{align*}
\P(Y>y|\,\bfX=\vect{x})=\sum_{k=1}^p\P(Y>y|J_0=k)\pi_k(\vect{x})\,,
\end{align*}
which is a mixture of $p$ PH distributions with different initial distributions, each assigning all its mass to a given state. In particular, we have the simple identity
\begin{align*}
\E[Y|\,\bfX=\vect{x}]=\sum_{k=1}^p\E[Y|J_0=k]\pi_k(\vect{x})\,.
\end{align*}
 For instance, in the homogeneous case we obtain
\begin{align}\label{mean_expression}
\E[Y|\,\bfX=\vect{x}]=\sum_{k=1}^p\pi_k(\vect{x})\bfe_k^{ \mathsf{T}}[-\bfT]^{-1}\bfe=\bfp(\vect{x})^{ \mathsf{T}} [-\bfT]^{-1}\bfe\,.
\end{align}
\end{remark}

\begin{example}\rm
If $D=\{\vect{x}_0\}$ is a singleton, then the PH-MoE model exactly spans the class of IPH distributions.
\end{example}

The following result shows that random covariates do not extend the marginal distribution beyond the above example.

\begin{proposition}
Let $\bfX$ be a random vector in a convex $D\subset \mathbb{R}^d$. Then the PH-MoE model has marginal distribution given by
\begin{align*}
\mbox{IPH}(\bfp(\vect{x}^\ast),\bfT,\lambda)\,,
\end{align*}
for some $\vect{x}^\ast\in D$. In fact, $\bfp(\vect{x}^\ast)=\E(\bfp(\bfX))$.
\end{proposition}
\begin{proof}
Let $\bfX$ have density $f:D\to\mathbb{R}_+$. We simply observe that by disintegration we get 
$$\P(J_0=k)=\int_D\pi_k(\vect{x})f(\vect{x})d\vect{x}=\pi_k(\vect{x}^\ast)\,,\quad k=1,\dots,p\,,$$ and since the process $(J_t)_{t\ge0}$ otherwise has the same dynamics after any initiation, the other parameters are unchanged. It remains to notice that $\vect{x}^\ast\in D$ by convexity.
\end{proof}

For the above reason, the PH-MoE model is most useful in its conditional form, and can be used for regression purposes. We now formulate a particularly advantageous parametrization for when $D=\mathbb{R}^d$.

\begin{definition}
We say that the PH-MoE  model with initial probabilities $\bfp(\bfX;\bfalp) = (\pi_k(\bfX;\bfalp))_{k = 1,\dots ,p}$ given by 
\begin{align}\label{eq:picov}
	\pi_k(\bfX;\bfalp) = \frac{\exp(\bfX^\mathsf{T} \bfalp_k)}{\sum_{j = 1}^{p} \exp(\bfX^\mathsf{T} \bfalp_j)}\,, \quad k = 1,\dots , p \,,
\end{align}
satisfies the softmax parametrization. Here, $\bfalp_k \in \overline{\mathbb{R}}^{d}$, $k = 1,\dots, p$, and $\bfalp = ( \bfalp_1^\mathsf{T}, \dots,  \bfalp_p^\mathsf{T})^\mathsf{T}\in \overline{\mathbb{R}}^{(p\times d)}$. 
\end{definition}

\begin{remark}\rm
In essence, we consider the coefficients of $\bfalp$ as assigning ``expertly" each observation to an initial distribution $\bfp(\bfX)$ according to their information $\bfX$. 

For the above parametrization, we have that the logarithm of the ratio between any two probabilities is linear in that for any $k, j\in\{1,\dots,p\}$,
\begin{align*}
\log\left(\frac{\pi_k(\bfX;\vect{\alpha})}{\pi_j(\bfX;\vect{\alpha})}\right)=\bfX^\mathsf{T}\left(\bfalp_k- \bfalp_j\right)=\sum_{i=1}^d X_i (\alpha_{ki}-\alpha_{ji})\,.
\end{align*}
However, for two individuals with covariate information $\bfX_1$ and $\bfX_2$, a class-specific correction arises, as follows:
\begin{align*}
\log\left(\frac{\pi_k(\bfX_1;\vect{\alpha})}{\pi_j(\bfX_2;\vect{\alpha})}\right)=\bfX_1^\mathsf{T}\bfalp_k- \bfX_2^\mathsf{T}\bfalp_j+\log\left(\frac{\sum_{i = 1}^{p} \exp(\bfX_2^\mathsf{T} \bfalp_i)}{\sum_{l = 1}^{p} \exp(\bfX_1^\mathsf{T} \bfalp_l)}\right).
\end{align*}
\end{remark}

In regression analyses, in particular, in the analysis of variance (ANOVA) it can often be the case that there is a discrepancy in the mean between observations belonging to two or more categories. However, the conditional distributions may not be Gaussian, or may even be different between different groups. We make a technical definition for such common situations.

\begin{definition}
Let $W_1,\dots,W_n$ be positive and continuous random variables having otherwise arbitrary distributions, and let $\eta\in\{1,\dots,n\}$ be a multinomial random variable, such that 
\begin{align*}
W_i {\ci} W_j\,,\:\: \forall i\neq j\,, \quad \mbox{and}\quad W_i {\ci}_{\bfX} \eta \,,\:\:\forall i\,,
\end{align*}
and such that $\bfX$ contains at least an intercept. Then we say that $W_\eta|\,\bfX$ follows a multinomial mixture distribution.
\end{definition}

\begin{proposition}\label{multi_mix_dens}
Let $W|\,\bfX$ follow a multinomial mixture distribution. Then there exist a sequence of PH-MoE models $(Y_m|\,\bfX)_{m\ge0}$ such that 
$$Y_m|\,\bfX\stackrel{d}{\to} W|\,\bfX\,, \quad m\to \infty\,.$$ 
Moreover, the softmax parametrization may be chosen.
\end{proposition}
\begin{proof}
We have by definition that $W|\,\bfX\stackrel{d}{=}W_\eta|\,\bfX$ for $W_1,\dots,W_n$ some conditionally independent variables, and a conditionally independent multinomial variable $\eta\in\{1,\dots,n\}$. Given $\bfX$, and by the denseness of PH distributions, there exist sequences of PH distributed random variables $Y_{im},$ $i=1,\dots,n$, $m=1,2,\dots$, such that, as $m\to \infty$, 
$$Y_{im}\stackrel{d}{\to}W_i\,,\quad i=1,\dots,n\,.$$ 
Since $\eta$ takes finitely many values, it follows that even  
$$Y_{\eta m}|\,\bfX\stackrel{d}{\to}W_\eta|\,\bfX,\quad m\to \infty\,.$$ 
It remains to note that for any given $m$, $Y_{\eta m}|\,\bfX$ is a finite mixture of independent PH variables, and thus PH distributed as well, with dimension $p^\ast$ at most the sum of the individual mixture-component dimensions. Finally, since the covariates contain an intercept term, the softmax function (as a function of $\bfalp$) 
\begin{align*}
	\pi_k(\bfalp) = \frac{\exp(\bfX^\mathsf{T} \bfalp_k)}{\sum_{j = 1}^{p^*} \exp(\bfX^\mathsf{T} \bfalp_j)},\quad k=1,\dots,p^\ast\,,
\end{align*} 
is a surjective map from $\mathbb{R}^{(d\times p^\ast)}$ to $\Delta^{p^\ast-1}$, and we may choose, for each $m$, $\bfalp$ to exactly match the required initial distribution of $Y_{\eta m}|\,\bfX$.
\end{proof}
\begin{remark}\rm
In the above result, any pre-specified tail behavior of $W|\,\bfX$ may be exactly matched by all the $Y_m|\,\bfX,$ $m=1,2,\dots$, by using the appropriate inhomogeneity function $\lambda$. The details are straightforward but technical and thus omitted.
\end{remark}

\section{Denseness on regression models}\label{sec:dens}

This section is devoted to showing a stronger version of Proposition \ref{multi_mix_dens}, under some more restrictive conditions on the covariate space and the associated conditional distributions.

\begin{definition}
Let $\mathcal{A}$ be the set of possible values of the covariates $\bfX$. A severity regression model is the set of conditional distributions of claim severity, given the covariates, that is, the set of laws of 
$$Y|\,\bfX=\vect{x}\,,\quad \vect{x}\in \mathcal{A}\,.$$ 
Given a severity regression model, we say that a sequence of severity regression models converges weakly (respectively, uniformly weakly) to it, if all the associated conditional distributions converge weakly for each $\vect{x}\in\mathcal{A}$ (respectively, uniformly weakly in $\vect{x}\in\mathcal{A}$).
\end{definition}

\begin{definition}
A feature space $\mathcal{A}$ is said to be regular if it is of the form $\mathcal{A}=\{1\}\times [a,b]^{d-1},$ $a,b\in\mathbb{R}$, that is, the covariates contain an intercept and are otherwise contained in a hypercube.
\end{definition}

\begin{condition}
A regression model is said to satisfy the tightness and Lipschitz conditions on $\mathcal{A}$ if
\begin{align*}
\{\P(Y\in \cdot\,|\,\bfX=\vect{x})\}_{\vect{x}\in\mathcal{A}}
\end{align*}
is a tight family of distributions, and for each $y\ge0$, the function
\begin{align*}
\vect{x}\mapsto\P(Y\le y\,|\,\bfX=\vect{x})
\end{align*}
is Lipschitz continuous in $\mathcal{A}$.
\end{condition}

To allow zeroes in the vector of initial probabilities, we can assume without loss of generality that the vector of initial probabilities is of the form \ $\bfp(\bfX;\bfalp) =(\tilde{\bfp}^{\mathsf{T}}(\bfX;\bfalp), \bf0)^{\mathsf{T}} $, where $\tilde{\bfp}(\bfX;\bfalp) = (\pi_k(\bfX;\bfalp))_{k = 1,\dots, q}$ is a $q$-dimensional column vector, $q\leq p$, with 
\begin{align*}
	\pi_k(\bfX;\bfalp) = \frac{\exp(\bfX^\mathsf{T} \bfalp_k)}{\sum_{j = 1}^{q} \exp(\bfX^\mathsf{T} \bfalp_j)}\,, \quad k = 1,\dots ,q \,.
\end{align*}
Indeed, we can always reorder the states of the PH representation in such a way that the first $q\leq p$ entries of $\bfp(\bfX;\bfalp) $ are the ones corresponding to the values different from zero.

\begin{proposition}[Denseness]\label{prop:dense}
Let a regression model satisfy the tightness and Lipschitz conditions on a regular $\mathcal{A}$. Then, there exists a sequence of PH-MoE regression models converging uniformly weakly to it.
	\end{proposition}

\begin{proof}
{\color{black} The claim follows from Theorem 3.3 in \cite{fung2019class} by noticing that LRMoE models with Erlang distributed severities (which satisfy Property 3 of Proposition 3.1 in \cite{fung2019class})  are particular instances of the PH-MoE model.}
\end{proof}

\begin{remark}\rm 
	Proposition~\ref{prop:dense} also implies that the PH-MoE models, with fixed inhomogeneity transformation $g$, form a dense class on the set of univariate severity regression distributions.
	
	This is relevant since it allows us to obtain different tail behaviors for modeling claim severities. For instance, in \cite{fung2019class}, it is shown that Pareto distributions fail to fulfill the  denseness conditions in the LRMoE model. 
	This implies that for a fixed splicing threshold, denseness and heavy-tails are not possible in that setting. In contrast, Pareto tail behavior is now possible using a PH-MoE model, while still preserving the denseness property. 
	
	 {\color{black} It is worth mentioning that another alternative was recently introduced in \cite{fung2020new}, which was termed the TG-LRMoE model.  The main idea of this model consists of transforming Gamma distributed random variables to obtain heavy tails for the severity distributions (including Pareto tails). Note, however, that when considering Gamma random variables with integer shape parameters (i.e., Erlang) in the TG-LRMoE specification, we obtain a particular case of a PH-MoE model with intensity $\lambda(y) = (1 + y)^{\gamma -1}$, $\gamma >0$.}
	
\end{remark}

\section{Estimation}\label{sec:est}

\subsection{The EM algorithm}

Suppose that we have a PH-MoE specification
\begin{align*}
Y|\bfX\sim\mbox{IPH}( \bfp(\bfX) , \bfT,\lambda)\,.
\end{align*}
By a simple inhomogeneity transformation, we may momentarily concentrate on the homogeneous case as follows:
\begin{align*}
Z|\bfX :=g^{-1}(Y|\bfX)\sim\mbox{PH}(  \bfp(\bfX) , \bfT )\,.
\end{align*}

Now, let $B_k(\bfX)$ be the number of times that the process $(J_t)_{t\geq0}$ with initial distribution $\bfp(\bfX)$ starts in state $k$, $N_{kl}(\bfX)$ the total number of jumps from state $k$ to $l$ conditional on $\bfp(\bfX)$, $N_k(\bfX)$ the number of times that we reach the absorbing state $p+1$ from state $k$ conditional on $\bfp(\bfX)$, and let $V_k(\bfX)$ be the total time that the underlying Markov jump process spends in state $k$ prior to absorption given $\bfp(\bfX)$.
Then, given a sample of absorption times $\vect{z}=(z_1,\dots,z_N)^{\mathsf{T} }$ and the corresponding paired covariate information $\overline{\vect{x}}=(\vect{x}_1,\dots,\vect{x}_N)$, the completely observed likelihood can be written in terms of the previously defined statistics as follows:
\begin{align*}
&\mathcal{L}_c( \bfp , \bfT |\vect{z},\overline{\vect{x}}) \\
&\quad =\prod_{i=1}^N\mathcal{L}_c( \bfp , \bfT | {z}_i,\bfX=\vect{x}_i)\nonumber\\
&\quad=\prod_{i=1}^N\prod_{k=1}^{p}{\pi_k(\vect{x}_i)}^{B_k(\vect{x}_i)} \prod_{k=1}^{p}\prod_{l\neq k} {t_{kl}}^{N_{kl}(\vect{x}_i)}\exp({-t_{kl}V_k(\vect{x}_i)})\prod_{k=1}^{p}{t_k}^{N_k(\vect{x}_i)}\exp({-t_{k}V_k(\vect{x}_i)}) \\
&\quad=\left(\prod_{i=1}^N\prod_{k=1}^{p}{\pi_k(\vect{x}_i)}^{B_k(\vect{x}_i)}\right) \prod_{k=1}^{p}\prod_{l\neq k} {t_{kl}}^{\sum_{i=1}^N N_{kl}(\vect{x}_i)}\exp\Big({-t_{kl}\sum_{i=1}^N V_k(\vect{x}_i)}\Big) \\
& \quad \quad  \times \prod_{k=1}^{p}{t_k}^{\sum_{i=1}^N N_k(\vect{x}_i)}\exp\Big({-t_{k}\sum_{i=1}^N V_k(\vect{x}_i)}\Big)\\
&\quad=\left(\prod_{i=1}^N\prod_{k=1}^{p}{\pi_k(\vect{x}_i)}^{B_k(\vect{x}_i)}\right) \prod_{k=1}^{p}\prod_{l\neq k} {t_{kl}}^{N_{kl}}\exp({-t_{kl}V_k})\prod_{k=1}^{p}{t_k}^{N_k}\exp({-t_{k}V_k})\label{complete_likelihood}\,,
\end{align*}
with
\begin{align*}
N_{kl}:=\sum_{i=1}^N N_{kl}(\vect{x}_i)\,,\quad V_k:=\sum_{i=1}^N V_k(\vect{x}_i)\,,\quad N_k:=\sum_{i=1}^N N_k(\vect{x}_i)\,.
\end{align*}
The above specification partially belongs to the exponential family of distributions and thus has semi-explicit maximum likelihood estimators.

Since the full-trajectory data is not observed, we employ the expectation-maximization (EM) algorithm to estimate part of the MLE iteratively. This implies that at each iteration, the conditional expectations of the sufficient statistics $B_k(\vect{x}_i)$, $N_{kl}$, $N_k$, and $V_k$ given the absorption times $\vect{z}$ are computed, corresponding to the E-step. Then  $\mathcal{L}_c( \bfp , \bfT ,\vect{z})$ is maximized by replacing the values of the statistics by their corresponding expected values from the previous step, obtaining in this way updated parameters $( \bfp , \bfT )$, commonly referred to as the M-step.

Then the detailed formulas are given as follows: 

\begin{enumerate} 
\item[ 1)]\textit{E-step, conditional expectations:} 
\begin{align*}
    \mathbb{E}(B_k(\vect{x}_i)\mid {Z}={z}_i,\bfX=\vect{x}_i)=\frac{\pi_k(\vect{x}_i) {\bfe_k}^{ \mathsf{T}}\exp( \bfT z_i) \bft }{ \bfpi(\vect{x}_i) \exp( \bfT z_i) \bft }\,,\quad i=1,\dots,N\,,
\end{align*}
\begin{align*}
 \mathbb{E}(V_k\mid \mat{Z}=\vect{z},\overline{\vect{x}})=\sum_{i=1}^{N} \frac{\int_{0}^{z_i}{\bfe_k}^{ \mathsf{T}}\exp( \bfT (z_i-u)) \bft  \bfpi(\vect{x}_i) \exp( \bfT u)\vect{e}_kdu}{ \bfpi(\vect{x}_i) \exp( \bfT z_i) \bft }\,,   
\end{align*}
\end{enumerate}

\begin{align*}
\mathbb{E}(N_{kl}\mid \mat{Z}=\vect{z},\overline{\vect{x}})=\sum_{i=1}^{N}t_{kl} \frac{\int_{0}^{z_i}{\bfe_l}^{\mathsf{T}}\exp( \bfT (z_i-u)) \bft \, \bfpi(\vect{x}_i) \exp( \bfT u)\vect{e}_kdu}{ \bfpi(\vect{x}_i) \exp( \bfT z_i) \bft }\,,
\end{align*}
\begin{align*}
\mathbb{E}(N_k\mid \mat{Z}=\vect{z},\overline{\vect{x}})=\sum_{i=1}^{N} t_k\frac{ \bfpi(\vect{x}_i)  \exp( \bfT z_i){\bfe}_k}{ \bfpi(\vect{x}_i) \exp( \bfT z_i) \bft }\,.
\end{align*}
\begin{enumerate}

\item[2)] \textit{M-step, explicit maximum likelihood estimators:} 
\begin{align*}
\hat t_{kl}=\frac{\mathbb{E}(N_{kl}\mid \mat{Z}=\vect{z},\overline{\vect{x}})}{\mathbb{E}(V_{k}\mid \mat{Z}=\vect{z},\overline{\vect{x}})} \,,\quad
\hat t_{k}=\frac{\mathbb{E}(N_{k}\mid \mat{Z}=\vect{z},\overline{\vect{x}})}{\mathbb{E}(V_{k}\mid \mat{Z}=\vect{z},\overline{\vect{x}})}\,,\quad \hat t_{kk}=-\sum_{l\neq k} \hat t_{kl}-\hat t_k \,.
\end{align*}

\item[3)] \textit{R-step, weighted multinomial regression estimation:} 
\begin{align*}
\hat\bfp(\cdot)=\argmax_{\bfp(\cdot)\in \Delta^{p-1}} \left(\prod_{i=1}^N\prod_{k=1}^{p}{\pi_k(\vect{x}_i)}^{\mathbb{E}(B_k(\vect{x}_i)\mid {Z}={z}_i,\bfX=\vect{x}_i)}\right),
\end{align*}
where as before $ \Delta^{p-1}$ is the standard $(p-1)$-simplex. 
\end{enumerate}

Finally, to incorporate the inhomogeneity transformation $g(\cdot)$, we assume that this is a parametric function depending on some vector $\bftheta$, that is, we consider $g(\cdot ; \bftheta)$. Then, $\bftheta$ is updated in a subsequent step consisting of direct maximization of the incomplete likelihood function with respect to (solely) this parameter. 

\begin{remark}\rm
In general, the set of all functions in the simplex is too broad, and a parametric family is chosen -- such as the softmax functions. Even then, no explicit solution for the R-step is available. We describe the entire procedure for the softmax case in Algorithm \ref{alg:IPHMoE}.

In view that the R-step is computed numerically even for the simplest logistic case, we see that Algorithm \ref{alg:IPHMoE} easily extends to the case where an arbitrary regression model with a categorical response is used to predict the initial Markov probabilities, for instance, when specifying $\vect{x}\mapsto\bfp(\vect{x})$ as a neural network. In this framework, the multinomial logistic regression model can be seen as a $0$-layer neural network.
\end{remark}

\begin{algorithm}[]
\caption{{EM algorithm for PH-MoE (Softmax parametrization)}}\label{alg:IPHMoE}
\begin{algorithmic}
\State \textit{\textbf{Input}: Positive data points $\vect{y}=(y_1,\dots,y_N)^{\mathsf{T}}$, covariates $\vect{x}_1,\dots,\vect{x}_N$, and initial parameters $( \bfalp , \bfT ,\bftheta)$.}\\
\begin{enumerate} 

\item[ 1)]\textit{Mixture specification:} Set 
\begin{align*}
	\pi_k(\vect{x}_i) = \pi_k(\vect{x}_i;\bfalp) = \frac{\exp(\vect{x}_i^\mathsf{T} \bfalp_k)}{\sum_{j = 1}^{p} \exp(\vect{x}_i^\mathsf{T} \bfalp_j)}\,,\quad i=1,\dots,N\,,\:\:k =1,\dots,p\,.
\end{align*}

\item[ 2)]\textit{Inhomogeneity transformation:} Transform the data into 
$$z_i=g^{-1}(y_i;  \bftheta )\,,\quad i=1,\dots,N\,.$$

\item[ 3)]\textit{E-step:} Compute the statistics
\begin{align*}
    \mathbb{E}(B_k(\vect{x}_i)\mid {Z}={z}_i,\bfX=\vect{x}_i)=\frac{\pi_k(\vect{x}_i) {\bfe_k}^{ \mathsf{T}}\exp( \bfT z_i) \bft }{ \bfpi(\vect{x}_i) \exp( \bfT z_i) \bft }\,,\quad i=1,\dots,N\,,
\end{align*}
\begin{align*}
 \mathbb{E}(V_k\mid \mat{Z}=\vect{z},\overline{\vect{x}})=\sum_{i=1}^{N} \frac{\int_{0}^{z_i}{\bfe_k}^{ \mathsf{T}}\exp( \bfT (z_i-u)) \bft  \bfpi(\vect{x}_i) \exp( \bfT u)\vect{e}_kdu}{ \bfpi(\vect{x}_i) \exp( \bfT z_i) \bft }\,,   
\end{align*}
\begin{align*}
\mathbb{E}(N_{kl}\mid \mat{Z}=\vect{z},\overline{\vect{x}})=\sum_{i=1}^{N}t_{kl} \frac{\int_{0}^{z_i}{\bfe_l}^{\mathsf{T}}\exp( \bfT (z_i-u)) \bft \, \bfpi(\vect{x}_i) \exp( \bfT u)\vect{e}_kdu}{ \bfpi(\vect{x}_i) \exp( \bfT z_i) \bft }\,,
\end{align*}
\begin{align*}
\mathbb{E}(N_k\mid \mat{Z}=\vect{z},\overline{\vect{x}})=\sum_{i=1}^{N} t_k\frac{ \bfpi(\vect{x}_i)  \exp( \bfT z_i){\bfe}_k}{ \bfpi(\vect{x}_i) \exp( \bfT z_i) \bft }\,.
\end{align*}

\item[4)] \textit{M-step: Let} 
\begin{align*}
\hat t_{kl}=\frac{\mathbb{E}(N_{kl}\mid \mat{Z}=\vect{z},\overline{\vect{x}})}{\mathbb{E}(V_{k}\mid \mat{Z}=\vect{z},\overline{\vect{x}})}\,,\quad
\hat t_{k}=\frac{\mathbb{E}(N_{k}\mid \mat{Z}=\vect{z},\overline{\vect{x}})}{\mathbb{E}(V_{k}\mid \mat{Z}=\vect{z},\overline{\vect{x}})}\,,\quad \hat t_{kk}=-\sum_{l\neq k} \hat t_{kl}-\hat t_k\,.
\end{align*}

\item[5)] \textit{R-step:} Maximize the weighted multinomial logistic regression
\begin{align*}
	\hat\bfalp=\argmax_{\bfalp\in\overline{\mathbb{R}}^{(p\times d)}}\sum_{i=1}^{N}\sum_{k = 1 }^{p} \mathbb{E}(B_k(\vect{x}_i)\mid {Z}={z}_i,\bfX=\vect{x}_i) \log(\pi_k(\vect{x}_i;\bfalp) )\,,
\end{align*} 
and set
\begin{align*}
	\hat\pi_k(\vect{x}_i) = \pi_k(\vect{x}_i;\hat\bfalp) = \frac{\exp(\vect{x}_i^\mathsf{T} \hat\bfalp_k)}{\sum_{j = 1}^{p} \exp(\vect{x}_i^\mathsf{T} \hat\bfalp_j)}\,,\quad i=1,\dots,N\,,\:\:k =1,\dots,p\,.
\end{align*}

\item[6)]  \textit{Inhomogeneity optimization:} Maximize
	\begin{align*}
	\hat{ \bftheta }
	&= \argmax_{ \bftheta } \sum_{i=1}^{N} \log \left( \lambda(y_i ;  \bftheta ) \hat{\vect{\pi}}^\mathsf{T}(\vect{x}_i) \exp\left({ \int_{0}^{y_i} \lambda(s ;  \bftheta )  ds \ \hat{\bfT} }\right) \hat{\bft} \right) \,.
	\end{align*}

\item[7)] Update the current parameters to $({\bfalp},\mat{T}, \bftheta ) =(\hat{{\bfalp}},\hat{\mat{T}}, \hat{ \bftheta })$. Return to step 1 unless a stopping rule is satisfied.
\end{enumerate}
    \State \textit{\textbf{Output}: Fitted representation $( \bfalp , \bfT ,\bftheta)$.}
\end{algorithmic}
\end{algorithm}

Direct calculations, or general results from EM theory, yield the following result.

\begin{proposition}
The likelihood function is increasing at each iteration of Algorithm \ref{alg:IPHMoE}. For a given $p$, the likelihood is also bounded, and we guarantee convergence to a (possibly local) maximum.
\end{proposition}

Notice that although convergence occurs, even if the parameters are such that the MLE of the PH distribution is asymptotically consistent, convergence to such MLE is still not guaranteed.

\subsection{Censoring}
In applications, an observation may be partially observed, in that only upper and/or lower bounds may be determined, but not its actual size. This incurs in a large bias if the bounds are far apart, and thus a statistical correction is required. Below we outline such adaptation to the estimation technique for PH-MoE models.

In essence, the EM algorithm~\ref{alg:IPHMoE} can be modified to work with censored observations, with just some adjustments on the formulas of the E-step being required. Recall that a data point is said to be right-censored at $a$ if it takes an unknown value above $a$, left-censored at $b$ if it takes an unknown value below $b$, and more generally interval-censored at $(a, b]$ if it takes an unknown value within the interval $(a, b]$. Moreover, note that for any censored observation of a PH-MoE model $ Y| \bfX \sim\mbox{IPH}( \bfp(\bfX) , \bfT,\lambda) $, the inhomogeneity transformation $g^{-1}(\cdot)$ results on a censored observation (of the same type) in the homogeneous setting $Z|\bfX = g^{-1}(Y |\bfX)\sim \mbox{PH}( \bfp(\bfX) , \bfT)$, meaning that we formally only need to deal with the latter case. 

In the following, we provide the explicit formulas for the E-step in the interval-censoring setting. Results for left and right censoring then follow as special cases, given that left-censoring can be seen as interval-censoring with $a = 0$ and right-censoring is retrieved by fixing $a$ and letting $b \to \infty$. 
Thus, for a single generic interval-censored observation $Z \in (a, b]$ with covariate information $\bfX=\vect{x}$, we have that
\begin{align*}
    \mathbb{E}(B_k(\vect{x})\mid Z \in (a, b]\,,\bfX=\vect{x}) &= \frac{\pi_k(\vect{x}) {\bfe_k}^{ \mathsf{T}}\exp( \bfT a) \bfe - \pi_k(\vect{x}) {\bfe_k}^{ \mathsf{T}}\exp( \bfT b) \bfe }{ \bfpi(\vect{x}) \exp( \bfT a) \bfe - \bfpi(\vect{x}) \exp( \bfT b) \bfe }\,,
\end{align*}

\begin{align*}
 & \mathbb{E}(V_k\mid Z\in (a, b],\bfX=\vect{x}) \\
 &\quad = \frac{1}{  \bfpi(\vect{x}) \exp( \bfT a) \bfe - \bfpi(\vect{x}) \exp( \bfT b) \bfe} \Bigg[  \int_{a}^{b}\bfpi(\vect{x}) \exp( \bfT u) \bfe_k du \\
 & \qquad \qquad \qquad \qquad \qquad \qquad \qquad \qquad -  \int_{0}^{b}{\bfe_k}^{ \mathsf{T}}\exp( \bfT (b-u)) \bft  \bfpi(\vect{x}) \exp( \bfT u)\vect{e}_kdu \\
& \qquad \qquad \qquad \qquad \qquad \qquad \qquad \qquad + \int_{0}^{a}{\bfe_k}^{ \mathsf{T}}\exp( \bfT (a-u)) \bft  \bfpi(\vect{x}) \exp( \bfT u)\vect{e}_k du\Bigg]\,,   
\end{align*}

\begin{align*}
& \mathbb{E}(N_{kl}\mid Z \in (a, b]\,,\bfX =  \vect{x}) \\
&\quad = \frac{t_{kl}}{  \bfpi(\vect{x}) \exp( \bfT a) \bfe - \bfpi(\vect{x}) \exp( \bfT b) \bfe } \Bigg[  \int_{a}^{b}\bfpi(\vect{x}) \exp( \bfT u) \bfe_k du \\
& \qquad \qquad \qquad \qquad \qquad \qquad \qquad \qquad- \int_{0}^{b}{\bfe_l}^{\mathsf{T}}\exp( \bfT (b-u)) \bft \, \bfpi(\vect{x}) \exp( \bfT u)\vect{e}_kdu \\
& \qquad \qquad \qquad \qquad \qquad \qquad \qquad \qquad+ \int_{0}^{a}{\bfe_l}^{\mathsf{T}}\exp( \bfT (a-u)) \bft \, \bfpi(\vect{x}) \exp( \bfT u)\vect{e}_kdu \Bigg]\,,
\end{align*}
\begin{align*}
\mathbb{E}(N_k\mid Z \in (a, b]\,,\bfX=\vect{x}) &= t_k\frac{ \int_{a}^{b} \bfpi(\vect{x})  \exp( \bfT u){\bfe}_k du}{  \bfpi(\vect{x}) \exp( \bfT a) \bfe - \bfpi(\vect{x}) \exp( \bfT b) \bfe }\,.
\end{align*}
The other steps of the algorithm are unchanged.

\subsection{Goodness of fit for phase-type regression models}\label{sec:goodness}
We propose a common visual tool for assessing the goodness of fit of the overall model. 
The procedure is described for the case when right-censored observations are present since it is the most common scenario in applications.
We define the residuals of a PH-MoE model by
\begin{align*}
r_i=-\log\left(\vect{\pi}(\bfx_i;\bfalp)\exp \left(\int_0^{y_i} \lambda (s;\, \vect{\theta})ds\ \mat{T} \right)\vect{e}\right)\,,\quad i=1,\dots, N \,,
\end{align*}
which under right-censoring completely at random and assuming (essentially, specifying the null-hypothesis) that the true distribution is indeed such PH-MoE, then by plugging in the estimated parameters from the EM algorithm, we obtain a dataset
\begin{align*}
\{(r_1\,\delta_1),\:(r_2,\,\delta_2),\dots, (r_N,\,\delta_N)\} \,,
\end{align*}
which follows a right-censored mean one exponential distribution. Here, $\delta_i$, $i = 1,\dots, N$, denote censoring indicators. In turn, we may construct a Kaplan-Meier survival curve for this dataset,
$ {S}(r)$, which should roughly resemble $S_0(r)=\exp(-r)$. To obtain a confidence band, we may use Greenwood's formula \
${\operatorname{Var}}({S}(r)) = {S}(r)^{2} \sum_{i: r_{i} \leq r} {d_{i}}/({n_{i}\left(n_{i}-d_{i}\right))}\,,$ where $d_i$ is the number of tied values at $r_i$, and $n_i$ all values yet to be observed (or at risk).




\section{Transforms}\label{sec:transforms}

The shape of the intensity function $\lambda$ is a central assumption of the PH-MoE model, which in particular determines the tail behavior, as can be deduced from \eqref{eq:IPHasymptotic}. This section introduces two useful global parametrizations for heavy-tailed distributions, and subsequently considers semi-composite models, which combine the conceptual approaches of splicing with our current setting.

Before introducing the parametric forms, we provide the exact tail behavior, which follows immediately from \eqref{eq:IPHasymptotic}.

\begin{proposition}\label{tails_moes}
Let $Y|\bfX$ be a PH-MoE specification. Let $\bfT({\bfX})$ be the sub-intensity matrix associated with the Markov jump-process $(J_{t})_{t\ge0}$ restricted to the accessible states $A(\bfX)\subset \{1,\dots,p\}$ when starting according to the distribution $\bfp(\bfX)$. Then
\begin{align*}
	\ov F_{Y|\bfX}(y|\vect{x}) \sim c(\vect{x}) [g^{-1}(y)]^{m(\vect{x}) -1} \exp({-\eta(\vect{x}) g^{-1}(y)}) \,, \quad y \to \infty \,,
\end{align*}
where $c(\vect{x})$ is a positive constant depending on $\bfp(\vect{x})$ and $\bfT(\vect{x})$, $-\eta(\vect{x})$ is the largest real eigenvalue of $\bfT(\vect{x})$, and $m(\vect{x})$ is the size of the Jordan block associated with $\eta(\vect{x})$. 

In particular, if $\bfp(\vect{x})$ never has zeros or if all states of the Markov process communicate, then $\bfT(\vect{x})=\bfT,$ for all $\vect{x}$, and all subgroups of the population have the same tail parameters.
\end{proposition}

\subsection{Global models}

Global models in the context of PH-MoE refer to parametrizations of $\lambda$ with respect to the same function on all of $\mathbb{R}_+$, as opposed to piece-wise functions. Such specifications are natural when considering the interpretation of the $g^{-1}$ function: it serves as a time transform that changes throughout time in a smooth way, that is, with continuous derivatives.

\subsubsection*{Pareto PH-MoE}
	Consider the transformation $$Y|\bfX= \theta (\exp({Z|\bfX})-1)\,,$$ where $Z|\bfX \sim \mbox{PH}(\bfp(\bfX) , \bfT )$ and $\theta > 0$. 
	Then, for $y\ge0$, 
	\begin{gather*}
		\ov{F}_{Y|\bfX}(y| \bfx ) = \bfpi(\bfx) \left(\dfrac{y}{\theta}+1\right)^{\bfT} \bfe \,,\\
		f_{Y|\bfX}(y| \bfx ) = \bfpi(\bfx) \left( \dfrac{y}{\theta}+1\right)^{\bfT - \mat{I}} \bft \, \frac{1}{\theta} \,.
	\end{gather*} 
	Here, $g(y)=\theta \left( \exp(y)-1\right)$ and $g^{-1}(y) = \log \left( {y}/{\theta}+1  \right)$. Consequently, the intensity function is given by 
\[ \lambda (y)  = \frac{1}{y+\theta}  \,.  \]

We refer to $Y|\bfX$ as a Pareto PH-MoE. It then follows from Proposition \ref{tails_moes} that $$\ov{F}_{Y|\bfX}(y| \bfx )\sim L(y,\bfx) y^{-\eta(\bfx)},$$ as $y \to \infty$, where $L(\cdot,\bfx)$ is a slowly varying function, that is, it satisfies that $\lim_{y \to \infty}L( c y,\bfx) / L( y,\bfx) = 1$ for all $c>0$, and $-\eta(\bfx)$ is the largest real eigenvalue of $\bfT(\bfx)$. The Pareto MoE  is designed to capture heavy-tailed (in the sense of regular variation) distributions with additional flexibility in the body of the distribution arising from the matrix parameters.

\subsubsection*{Weibull PH-MoE}
	If we now instead consider $$Y|\bfX= {(Z|\bfX)}^{1/\theta},$$ where $Z|\bfX \sim \mbox{PH}(\bfp(\bfX) , \bfT )$ and $\theta > 0$, then for $y\ge0$,
	\begin{gather*}
		\ov{F}_{Y|\bfX}(y| \bfx ) = \bfpi(\bfx) \exp({\bfT y^{\theta}}) \bfe \,, \\
		f_{Y|\bfX}(y| \bfx ) = \bfpi(\bfx) \exp({\bfT y^{\theta}}) \bft \, \theta y^{\theta-1} \,.
	\end{gather*} 
Hence, $g(y)=y^{1/\theta}$, $g^{-1}(y)=y^\theta $, and 
\[ \lambda (y) = \theta y^{\theta-1} . \]
We refer to this model as a Weibull PH-MoE, where loosely speaking we obtain, for each observation, a Weibull tail behavior with a matrix in place of the usual scale parameter. From  Proposition \ref{tails_moes}, it follows that 
\begin{align}\label{tail_weibullmoe}
\ov{F}_{Y|\bfX}(y| \bfx ) \sim c(\bfx) y^{\gamma(\bfx)} \exp({-\eta(\bfx) y^{\theta}}),
\end{align} as $y \to \infty$, where $c(\bfx) > 0$, $\gamma(\bfx) \geq 0$, and $-\eta(\bfx)$ is as above. 

This model is suitable for a wider range of applications, since it falls into the Gumbel max-domain of attraction, which implies that it has strictly lighter tails than those of Pareto-type. However, for $\theta <1$ (respectively, $\theta>1$), we get that \eqref{tail_weibullmoe} specifies heavier (respectively, lighter) tails than exponentially decaying ones.

An interesting feature of this specification, and contrary to the Pareto case, is that conditional means are fully explicit and given by
\begin{align*}
	\E(Y^{\zeta}|\bfX)= \Gamma(1+ \zeta/\theta) \bfpi(\bfX) (- \bfT)^{-\zeta/\theta} \bfe \,\quad \forall \zeta>0\,.
\end{align*}

\begin{example}[Different tail behaviors]\rm
We illustrate the importance of Proposition \ref{tails_moes} by providing a simple two-groups case where different tail behavior arises. Consider the matrix
\begin{align*}
	{\bfT}=\left(\begin{matrix}{}
  -1& 0.5& 0 \\
 1 &-2 & 0\\
 0 &0 & -3 \\
\end{matrix}\right),
\end{align*}
and two groups with initial distributions
 \begin{align*}
 \bfp(\mbox{Group 1})&= (1,\: 0,\: 0),\\
  \bfp(\mbox{Group 2})&= (0,\: 0,\: 1),
 \end{align*}
 respectively. Then the first group can only access the first two states, and thus its tail is of the order $\color{black}\exp(-0.634 y)$ (since $-0.634$ is the largest eigenvalue of the sub-matrix $\color{black}(t_{kl})_{k,l=1,2}$), while the second group can only access the third state, and thus has a tail  of order $\color{black}\exp(-3y)$.
 
Similarly, if a Pareto inhomogeneity function is used, the tail index will vary between the two groups. Thus, after estimation, it is important to check which states are accessible by which sub-populations, in order to deduce their precise conditional tail asymptotics.
\end{example}

\subsection{Semi-composite models as an alternative to splicing}
When heavy tails are present, a standard approach to obtain a global model for claim severities is to model the body and tail separately, and then combine them through splicing or mixing. For describing the tail of the distribution, extreme value tools are typically employed. Although this two-step procedure is not fully satisfactory, the outcome can be more reliable when the parameter of interest is the tail coefficient, see, for instance, \cite{embrechts2013modelling}. 

On the other hand, the models presented in the previous section are attractive alternatives to obtain global models, due to their authentic heavy tails and denseness. Moreover, the fitting of these models does not require any form of threshold selection, as in traditional extreme value techniques. 
 However, their estimation methods give the same weight to all data points, and hence the automatic modeling of the tails may not be as satisfactory as when targeting the tail via thresholding. Furthermore, in some situations, even if the tail is correctly specified via fitting a PH-MoE model using the EM algorithm, a risk manager might be interested in at least partially separating the analysis above and below a certain threshold. 
 
 Below we see how certain piecewise specifications for the inhomogeneity function $\lambda$ can achieve a compromise between the two above approaches, while still formally falling into the class of standard PH-MoE models.
Specifically, we consider inhomogeneity transformations which are defined differently below and above a certain threshold (and the idea can be extended to several layers).

\begin{definition}
We say that a PH-MoE model is semi-composite if its intensity function is of the form 
\begin{align*}
	\lambda(t) = \left\{\begin{matrix}{}
   \lambda_1(t) \,, & t \leq y_0 \,,\\
   \lambda_2(t)\,, &  t > y_0\,,
\end{matrix} \right.
\end{align*}
for any two intensities $\lambda_1,\lambda_2$. 
\end{definition}
An immediate consequence is the following:
\begin{proposition}
For a semi-composite PH-MoE model we have that
\begin{align*}
	g^{-1}(y) = \left\{\begin{matrix}{}
   g^{-1}_1(y)\,, & y \leq y_0 \,,\\
   g^{-1}_2(y)+g^{-1}_1(y_0) -g^{-1}_2(y_0)\,, &  y > y_0 \,,
\end{matrix} \right.
\end{align*}
and so in particular
\begin{align*}
		\ov{F}_{Y|\bfX}(y| \bfx ) =  \left\{\begin{matrix}{}
   \bfpi(\bfx) \exp(\bfT g_1^{-1}(y) )\bfe  \,, & y \leq y_0 \,,\\
   \bfpi(\bfx) \exp( (g^{-1}_2(y)+g^{-1}_1(y_0) -g^{-1}_2(y_0)) \bfT)\bfe \,,&  y > y_0\,.
\end{matrix} \right.
\end{align*}
Hence, $Y|\bfX$ is tail-equivalent to a PH-MoE model with intensity $\lambda_2(t)$ for all $t\ge0$.
\end{proposition}
Below we outline the details of two cases which give rise to tails which are commonly used for loss modeling. 

\begin{example}[PH body with Weibull tail]\normalfont
	Specify 
	\begin{align*}
	\lambda(t) = \left\{\begin{matrix}{}
   1 \,, & t \leq y_0 \,,\\
   \theta (t - y_0) ^{\theta -1}\,, &  t > y_0\,.
\end{matrix} \right.
\end{align*}
In this way 
\begin{align*}
	g^{-1}(y) = \left\{\begin{matrix}{}
   y\,, & y \leq y_0 \,,\\
   y_0 + (y - y_0) ^{\theta}\,, &  y > y_0 \,,
\end{matrix} \right.
\end{align*}
and 
\begin{align*}
	g(y) = \left\{\begin{matrix}{}
   y\,, & y \leq y_0 \,,\\
   y_0 + (y - y_0) ^{1/\theta} \,, &  y > y_0\,.
\end{matrix} \right.
\end{align*}
Hence
\begin{align*}
		\ov{F}_{Y|\bfX}(y| \bfx ) =  \left\{\begin{matrix}{}
   \bfpi(\bfx) \exp(\bfT y )\bfe  \,, & y \leq y_0 \,,\\
   \bfpi(\bfx) \exp( (y_0 + (y - y_0) ^{\theta}) \bfT)\bfe \,,&  y > y_0\,.
\end{matrix} \right.
\end{align*}
\end{example}

\begin{example}[PH body with Pareto tail]\normalfont
	We may instead specify 
	\begin{align*}
	\lambda(t) = \left\{\begin{matrix}{}
   1 \,,& t \leq y_0 \,,\\
    (t - y_0 + \theta) ^{-1} \,, &  t > y_0 \,.
\end{matrix} \right.
\end{align*}
Which now yields
\begin{align*}
	g^{-1}(y) = \left\{\begin{matrix}{}
   y \,, & y \leq y_0 \,,\\
   y_0 + \log((y - y_0) / \theta + 1)  \,, &  y > y_0\,,
\end{matrix} \right.
\end{align*}
and 
\begin{align*}
	g(y) = \left\{\begin{matrix}{}
   y \,,& y \leq y_0 \,,\\
   y_0 + \theta( \exp(y - y_0) - 1) \,, &  y > y_0 \,.
\end{matrix} \right.
\end{align*}
Hence
\begin{align*}
		\ov{F}_{Y|\bfX}(y| \bfx ) =  \left\{\begin{matrix}{}
   \bfpi(\bfx) \exp(\bfT y )\bfe  \,, & y \leq y_0 \,,\\
   \bfpi(\bfx) \exp( y_0 \bfT) \left( \frac{y - y_0}{\theta} + 1\right)^{\bfT} \bfe \,, &  y > y_0 \,.
\end{matrix} \right.
\end{align*}
\end{example}

Note that for fix $\bfx$, the proposed models are dense in the class of distributions in the positive real line. This follows from the fact that they belong to the IPH class, which possesses the said property for any $\lambda$ satisfying \eqref{lambda_restrictions}.

Given that the intensity function $\lambda$ is a parametric function depending on the parameters $\theta$ and $y_0$, we can employ Algorithm \ref{alg:IPHMoE} for the estimation of the above semi-composite specifications. 
However, the changepoint may alternatively be specified in advance and then fixed through the fitting procedure.
The latter approach is preferable in almost all cases, and in particular when working with regularly-varying heavy tails, where the threshold may be determined by well-founded visual tools, such as the Hill estimator, cf. \cite{hill1975simple}. 

%
%
%




\section{Numerical examples}\label{sec:examples}
This section illustrates the statistical feasibility of the methods developed above. We do not aim to be comprehensive in our treatment, but instead point out the general direction which seems promising. The drawback of the algorithm at the moment is speed, in particular of the R-step, which is a well-known issue of multinomial regression models when in the presence of several covariates. Hence, we provide one simulated example with a $4$-dimensional categorical covariate, and a much larger real insurance example with two categorical covariates.

\subsection{Synthetic data}
We consider a simulated example where the data genuinely comes from a classical mixture-of-experts model. More precisely, the dataset consists of a total of $2,000$ observations divided into $4$ groups of size $500$, each having  distributions as follows: 
\begin{align*}
&\mbox{Group A: } Y_i\sim \Gamma(\mbox{shape}=1,\mbox{scale}=3),\quad \mbox{Group B: } Y_i\sim \Gamma(\mbox{shape}=3,\mbox{scale}=9),\\
&\mbox{Group C: } Y_i\sim \Gamma(\mbox{shape}=1,\mbox{scale}=9),\quad\mbox{Group D: } Y_i\sim \Gamma(\mbox{shape}=3,\mbox{scale}=3).
\end{align*}
We consider a $5$-dimensional PH structure, which was chosen small enough so that if there were no interaction between states, it would not be possible to model the four groups. Indeed, mixtures of five exponential components would not correctly capture the four given distributions.

Subsequently, we employed a homogeneous version of Algorithm \ref{alg:IPHMoE}, where Steps 2 and 6 are suppressed. In other words, we assume that $\lambda(t)=1$, $\forall t\ge0$, and in particular, the tails of both the data and the model are exponentially decaying. The results are as follows\footnote{Here and in the rest of the numerical section, estimates are rounded to three decimal places. This means that some displayed null values may actually be very small but non-zero.}:
\begin{align*}
	\hat{\bfT}=\left(\begin{matrix}{}
  -0.349& 0& 0 &0&0\\
 0.303 &-0.303 & 0 &0&0\\
 0 &0.162 & -0.553&0&0.391 \\
  0& 0&0.059 &-0.06 & 0.001\\
    0& 0.618&0.607 &0 & -1.225
\end{matrix}\right),
\end{align*}
and the output of the R-step is given in Table \ref{reg:coef}.
\begin{table}[!htbp] \centering 
\begin{tabular}{@{\extracolsep{5pt}}lcccc} 
\hline \\[-1.8ex] 
State & \multicolumn{1}{c}{(Intercept)} & \multicolumn{1}{c}{Group B} & \multicolumn{1}{c}{Group C} & \multicolumn{1}{c}{Group D} \\ 
\hline \\[-1.8ex] 
2 & -9.986 (6.673) & 17.488 (24.118) & 8.732 (6.675)& 13.983 (30.537) \\ 
3 & -4.32 (0.395)& -3.113 (0.044) &  4.308  (0.41)& 17.281 (29.532)\\ 
4 & -12.642 (25.17) & 25.136 (34.211) & 12.002 (25.17) &  8.02  (300.3)\\ 
5 & -4.488 (0.429) & 6 (25.6) &   3.079 (0.464) &  13.74 (29.534)\\ 
\hline \\[-1.8ex] 
\end{tabular} 
  \caption{Regression coefficients $\hat{\bfalp}$ for the categorical variable of the simulated data. Group A is the baseline level. In parenthesis, the standard errors are displayed. The coefficient associated with state $1$ can be deduced from the constraint $\sum_{k=1}^5 \pi_k(\bfX)=1$.} 
  \label{reg:coef} 
\end{table} 
  
 In particular, the coefficients translate into the following initial Markov probabilities for each group:
 \begin{align*}
 \bfp(\mbox{Group A})&= (0.976,\: 0.000,\: 0.013,\: 0.000,\: 0.011),\\
  \bfp(\mbox{Group B})&= (0.000,\: 0.007,\: 0.000,\: 0.993,\: 0.000 ),\\
 \bfp(\mbox{Group C})&= (0.328,\: 0.094,\: 0.324,\: 0.173,\: 0.080 ),\\
 \bfp(\mbox{Group D})&= (0.000,\: 0.000,\: 0.976,\: 0.000,\: 0.024 ).\\
 \end{align*} 

When considering the means for each group, given by formula \eqref{mean_expression}, we obtain, upon comparing with a Gamma Generalized Linear Model (GLM), the following table:
\begin{table}[!htbp] \centering 
\begin{tabular}{@{\extracolsep{5pt}}lccc} 
\hline \\[-1.8ex] 
Group& \multicolumn{1}{c}{Theoretical}  & \multicolumn{1}{c}{Empirical ($=$GLM)} & \multicolumn{1}{c}{PH-MoE} \\ 
\hline \\[-1.8ex] 
A&3 & 3.005  & 3.021 \\ 
B&27 & 27.212 & 26.347\\ 
C&9 &  9.463 & 10.001 \\ 
D&9 & 9.499 & 9.807 \\
\hline \\[-1.8ex] 
\end{tabular} 
  \caption{Theoretical, observed, and fitted means.} 
  \label{reg:means} 
\end{table} 
  
Since the PH-MoE does not match the empirical means for each group, as is the case for the GLM, we can observe slight discrepancies between the fitted and observed averages per group. However, if we estimate the dispersion coefficient of the GLM with the average deviance, we may compare not only the mean but the entire distribution of both models. Figure \ref{fit_synth} shows the densities for each group for the theoretical and fitted cases, and for the GLM and PH-MoE models. We observe that the risks are better understood if we use the latter model, and consequently, any other measure of performance which is not solely based on the mean will favor the matrix-based method.

\begin{figure}[!htbp]
\centering
\includegraphics[width=0.7\textwidth]{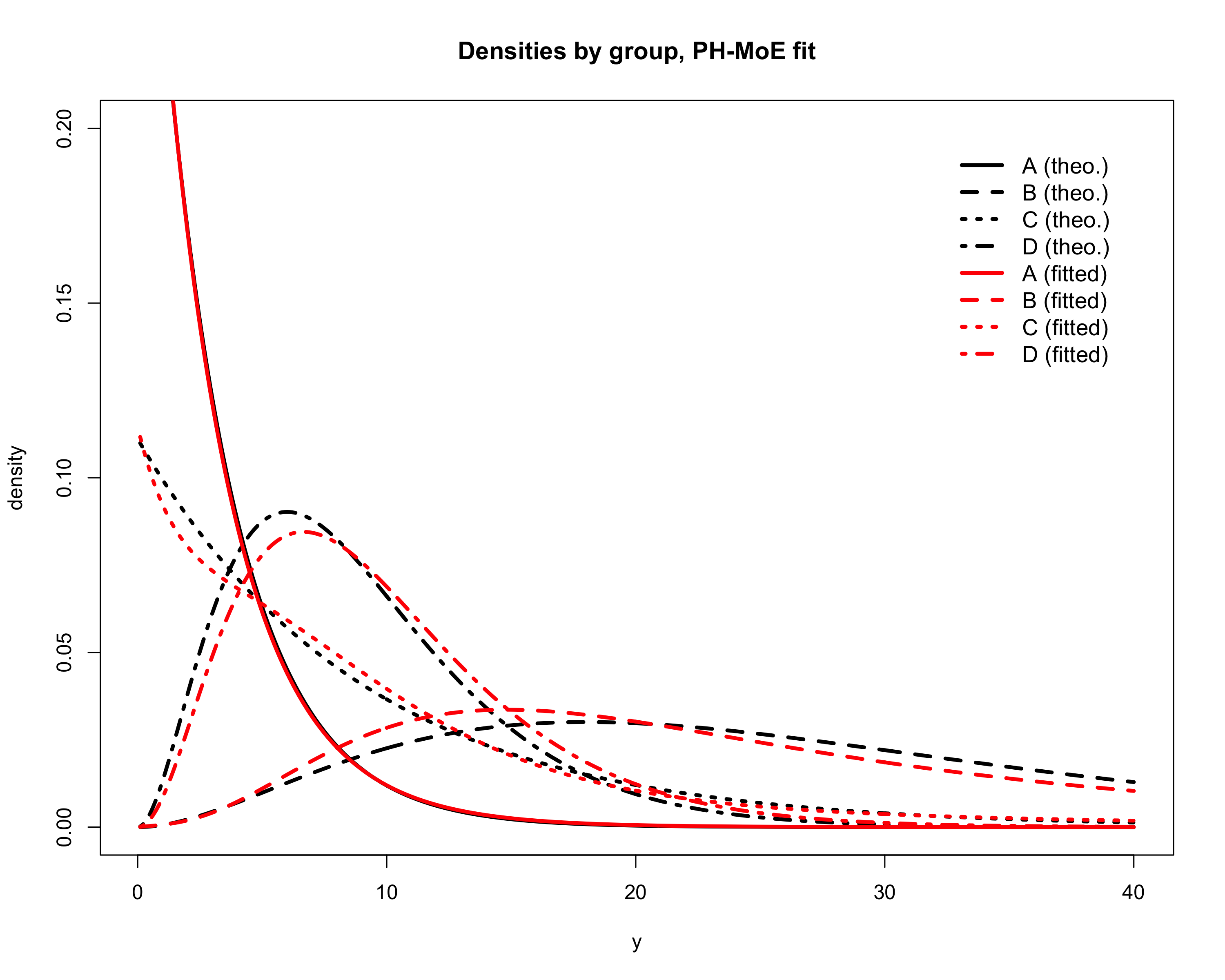}
\includegraphics[width=0.7\textwidth]{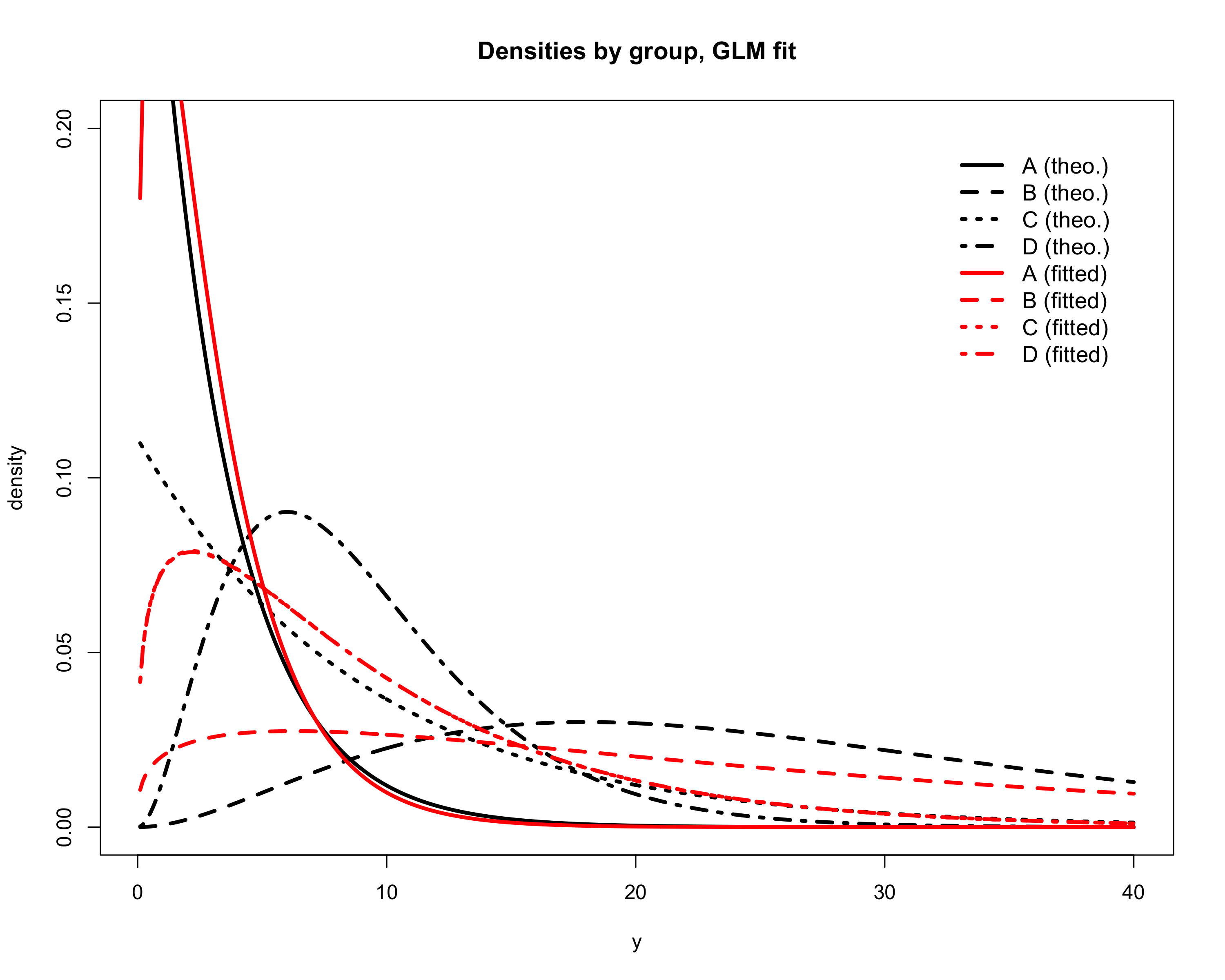}
\caption{Fitted densities for each group of the simulated data, for the PH-MoE (top panel) and GLM (bottom panel) models.
} \label{fit_synth}
\end{figure}

\subsection{Insurance data}

We consider the French Motor Third Party Liability \linebreak (freMTPL) insurance data, contained in the datasets \texttt{freMTPLfreq} and \texttt{freMTPLsev} in the \texttt{CASdatasets} package in \texttt{R}. The data consists of risk features corresponding to $413,169$ motor insurance policies and the number of claims and their severity. Presently we aim at analyzing only a portion of the total $15,390$ claim sizes\footnote{Here, we have divided claim severities by the corresponding claim numbers for each policy. For numerical reasons, we also divided the result by $10^4$.}, mainly for computational power reasons: the multinomial step of the PH-MoE routine can be slow to (or not) converge for large $p$ (say, above $10$), $n$ (in the tens of thousands) and $d$ (more than $20$ covariates). 

To find a sensible and interesting subset of the data, we first observe the top left panel of Figure \ref{hist_full_data}, where we see that the data has a highly pronounced peak in the log scale. Such peak can only be captured by a PH distribution of a huge order ($p>200$), which makes it unfeasible to fit even without covariates. Thus, we consider only the excesses above the threshold $M=0.15$, which in insurance terms would correspond to data entering an XL reinsurance contract with retention level $M$. 
The excesses are plotted in the top right panel of Figure \ref{hist_full_data} in the log scale, which are much easier to estimate with a lower dimension. 
On the bottom panel of Figure \ref{hist_full_data}, we observe the heavy-tailed nature of the excesses and that there is a substantial bias away from strict Pareto behavior (the estimator curves for smaller order statistics).
\begin{figure}[!htbp]
\centering
\includegraphics[width=0.49\textwidth]{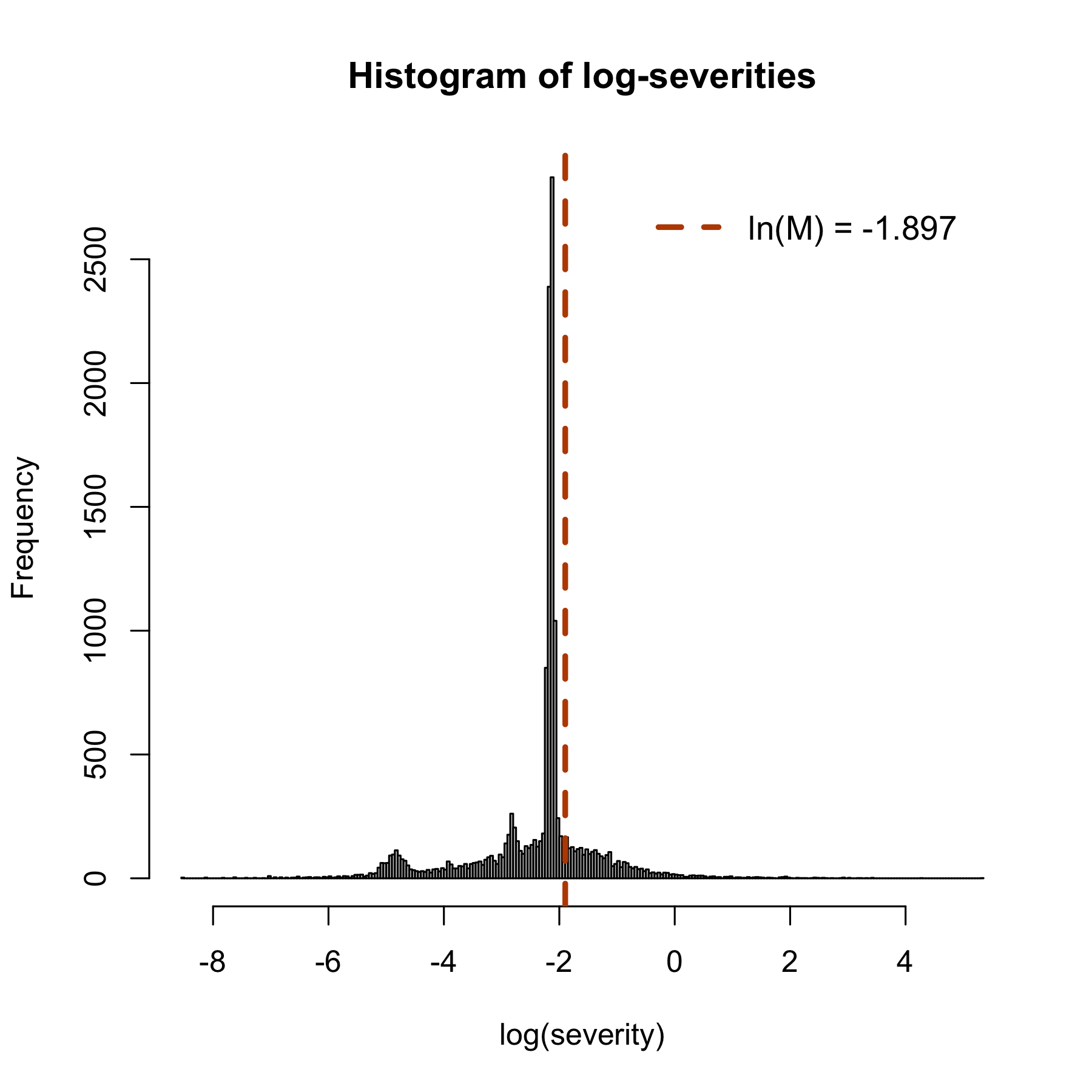}
\includegraphics[width=0.49\textwidth]{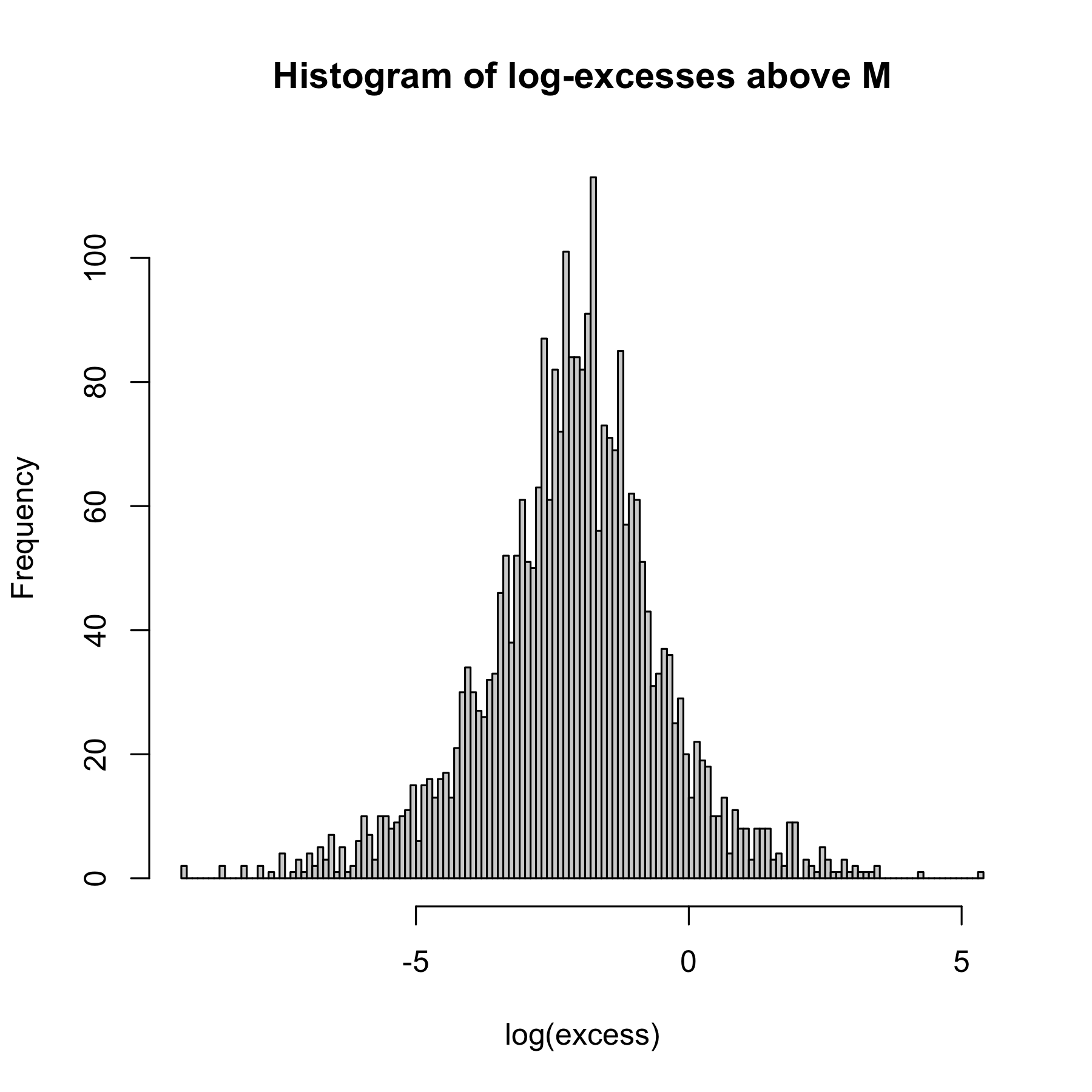}
\includegraphics[width=0.9\textwidth]{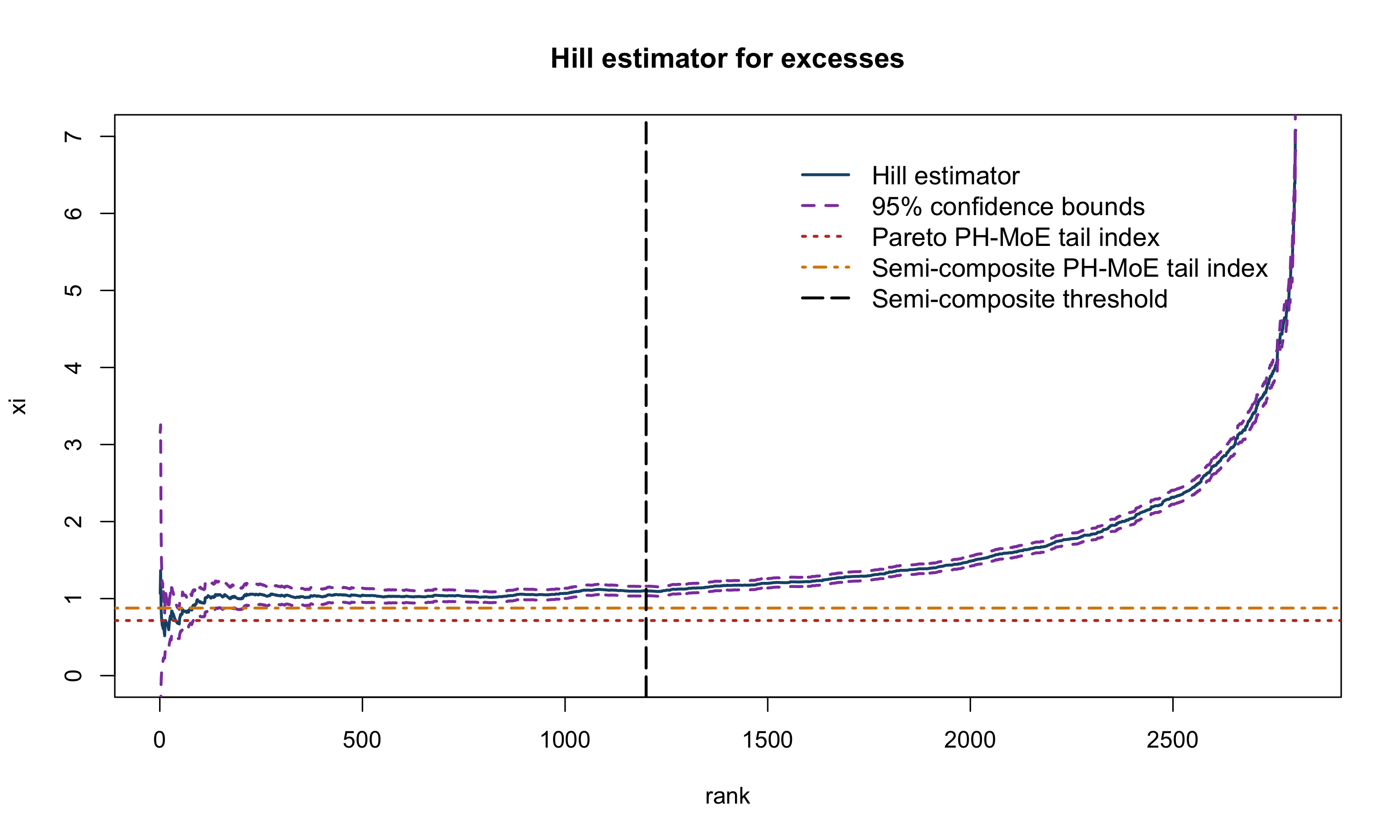}
\caption{French MTPL full data with selected excess threshold (top left panel), the resulting excesses (top right), and their implied tail index according to the Hill estimator (bottom). For the latter plot, we also overlay the implied tail indices from two different PH-MoE fits.
} \label{hist_full_data}
\end{figure}

With respect to covariates, we fitted a log-normal regression model and selected the only two covariates which seem to be relevant to predicting the mean\footnote{In general, insurance covariates provide very small predictive power for severity, in contrast to claim counts, where the performance is usually much better. }: 
\begin{enumerate}
\item \texttt{Power}: The power of the car, an ordered categorical variable with values: d, e, f, g, h, i, j, k, l, m, n,  o.\\
\item \texttt{Region}: The policy region in France, based on the 1970-2015 classification. Possible values associated with the excesses are: Aquitaine,     Basse-Normandie,   
Bretagne,          Centre,            
 Haute-Normandie,   Ile-de-France,   
 Limousin,          Nord-Pas-de-Calais
 Pays-de-la-Loire,   Poitou-Charentes.
\end{enumerate}

{\color{black}

To illustrate and compare the modeling capabilities of our model, we proceed to estimate different PH-MoE and LRMoE models, the latter being a natural candidate for comparison. More specifically, and to keep the number of parameters for both models similar, we considered three Pareto PH-MoE models of dimensions 3, 4, and 5, and three LRMoE models with 4, 5, and 6 experts.
For the PH-MoE models, we employed $1000$ EM steps with random initialization of the parameters. On the other hand, the estimation of the LRMoE models was done using the LRMoE \texttt{R} package (cf. \cite{tseung2020lrmoe}) and as per the Vignettes found in \url{https://github.com/UofTActuarial/LRMoE/tree/master/vignettes} with $200$ CEM steps and expert components automatically chosen with the $\texttt{cmm}$\_$\texttt{init}$ function.
 A summary of the results can be found in Table~\ref{tab:phmoe} and Table~\ref{tab:lrmoe}.

\begin{table}[!htbp]
\centering
\begin{tabular}{l |c c c }

  & \multicolumn{3}{ c  }{PH-MoE}     \\
\hline
 Dimension & $3$   & $4$   & $5$ \\
 \hline 
Log Likelihood & $718.38 $   & $743.52$   & $759.71$   \\
Number of parameters & $52$   & $80$   & $110$   \\
Computational times            & 3.14 mins   & 6.24 mins   & 11.88 mins \\
\hline
\end{tabular}
\caption{Summary for PH-MoE model for the freMTPL dataset.}
\label{tab:phmoe}
\end{table}

\begin{table}[!htbp]
\centering
\begin{tabular}{l |c c c }

 & \multicolumn{3}{ |c  }{LRMoE}    \\
\hline
Number Experts  & $4$   & $5$   & $6$\\
 \hline 
Log Likelihood & $728.64 $   & $746.64 $   & $750.29$   \\
Number of parameters & $71$   & $94$   & $117$   \\
Computational times        & 49.75 mins   & 1.21 hours  & 1.65 hours \\
\hline
\end{tabular}
\caption{Summary for LRMoE model for the freMTPL dataset.}
\label{tab:lrmoe}
\end{table}

Although perhaps mathematically more complex, our investigations show that the numerical routines for PH-MoE models are at least on par with those of LRMoE in terms of likelihood performance, and certainly much faster. Note also that the number of EM steps cannot be compared with those of the CEM algorithm since the latter is much slower but converges in fewer iterations. This is in line with the thinking of the additional parameters of a PH-MoE model as weak learners rather than actual statistical parameters. A full systematic comparison between PH-MoE, LRMoE, TG-LRMoE, and related MoE models in terms of in-sample and out-of-sample performance is out of the scope of this work. However, we can mention that an advantage of having better computational times, is that we can try different initializations for the EM algorithm. This is highly relevant for the estimation of both models since there is always the possibility of obtaining a local maxima depending on the initial values.

It can be appreciated that the number of parameters involved in both models is relatively large.  
 However, it is crucial to understand that these models can be considered as interpretable machine learning methods rather than concise statistical models. In other words, each additional degree of freedom does not always target a particular distributional feature but instead serves as a \textit{weak learner}, working towards an overall good estimation. Hence, classical information criteria such as AIC and BIC will tend to overpenalize these models and should not be used for model selection.
 {\color{black} Ideally, an information criterion specifically designed for PH-MoE models (and even for PH variables) would allow for goodness of fit considerations without using the misspecified AIC and BIC criteria, similar to the development of AICC in the context of time series analysis.} Regularization through cross-validation is {\color{black}also} a natural topic of further study, but out of the scope of the current manuscript.

With the different PH-MoE models above at hand, we now select one to describe our data and present the results. We start by giving some words about the dimension selection of the PH-MoE model, for which we will follow the typical approach of dimension selection for PH distributions. More specifically, given that adding more dimensions always improves the quality of the fit, one typically starts with low dimensions and assesses the benefit/cost of adding extra dimensions. This assessment is usually done using visual aids and/or by looking at changes in the loglikelihood. One aims for a dimension that is a good compromise between the quality of the fit and a reasonable number of parameters. The reason for this approach is the identifiability issues of PH distribution, meaning that the number of free parameters is unknown. However, it is worth mentioning that recent steps towards more statistical-based selection approaches have been recently introduced in the literature. For instance, we can mention the work in \cite{penalised}.  

In our particular case of study, dimension 5 seems to be a good compromise. To support our choice, we also fitted a PH-MoE model of dimension 6, obtaining an increase in the loglikelihood of $8.15$, and additional 32 parameters, which was a much smaller likelihood increase than the previous steps. Hence, we decided to stick to a PH-MoE of dimension 5. 
For completeness, we also fitted a semi-composite PH-MoE with Pareto tail and same dimension 5. In this case, we select the threshold value $y_0$ to be at the order statistic number $1200$, which is when the Hill plot visually starts to flatten, obtaining a loglikelihood of $759.3$. It is worth mentioning that we are selecting the threshold but not the implied tail index. The latter is estimated during the EM algorithm jointly with all other parameters. 

The full fitted coefficients $\hat{\bfalp}$ and their statistical significance are given in Appendix \ref{apA}, with the asymptotic properties of the MLE delegated to Appendix \ref{sec:inference}. In Figure \ref{two_densities}, we observe that the densities shift their shapes when varying the covariates, as expected. The semi-composite model has a discontinuity at $y_0$, which is slightly visible, a feature that is also common in fully composite models. The associated PH-MoE probabilities for these two densities are as follows:
 \begin{align*}
 \mbox{Pareto PH-MoE: } &\bfp(\mbox{Power f, Region Centre})= (0.000,\:  0.000,\:  0.000,\:  0.138,\:  0.862 ),\\
 &\bfp(\mbox{Power g, Region Centre})= (0.074,\: 0.014,\: 0.695,\: 0.132,\: 0.085),\\
&\hat{\bfT}=\left(\begin{matrix}{}
-22.119  & 0.000 & 0.005 & 3.580  & 0.041\\
0.011 & -9.233 & 6.689 & 0 &  2.511 \\
0 & 0.292 & -9.931 & 0 &  0.518\\
0.022 & 1.285 & 0.064 & -1.402 &  0.026\\
0.005 & 0.404 & 0.986 & 0 & -13.203\\
\end{matrix}\right),\\
&\hat{\theta}=1.639.
  \end{align*} 
  \begin{align*}
 \mbox{S-C PH-MoE: } &\bfp(\mbox{Power f, Region Centre})= (0.000,\: 0.000,\: 0.000,\: 0.153,\: 0.847),\\
 &\bfp(\mbox{Power g, Region Centre})= (0.104,\: 0.023,\: 0.655,\: 0.138,\: 0.079 ).\\
 &\hat{\bfT}=\left(\begin{matrix}{}
-14.444 & 0.000 & 0.008 & 2.639 & 0.053\\
0.007 & -5.734 & 4.569 & 0 & 1.146 \\
0 & 0.086 & -5.785 & 0 & 0.479\\
 0.006 & 1.103 & 0.024 & -1.142 &  0.008\\
0.003 & 0.148 & 1.107 & 0 & -8.351\\
\end{matrix}\right),\\
&\hat{\theta}=0.961 .
 \end{align*} 

}

\begin{figure}[!htbp]
\centering
\includegraphics[width=0.7\textwidth]{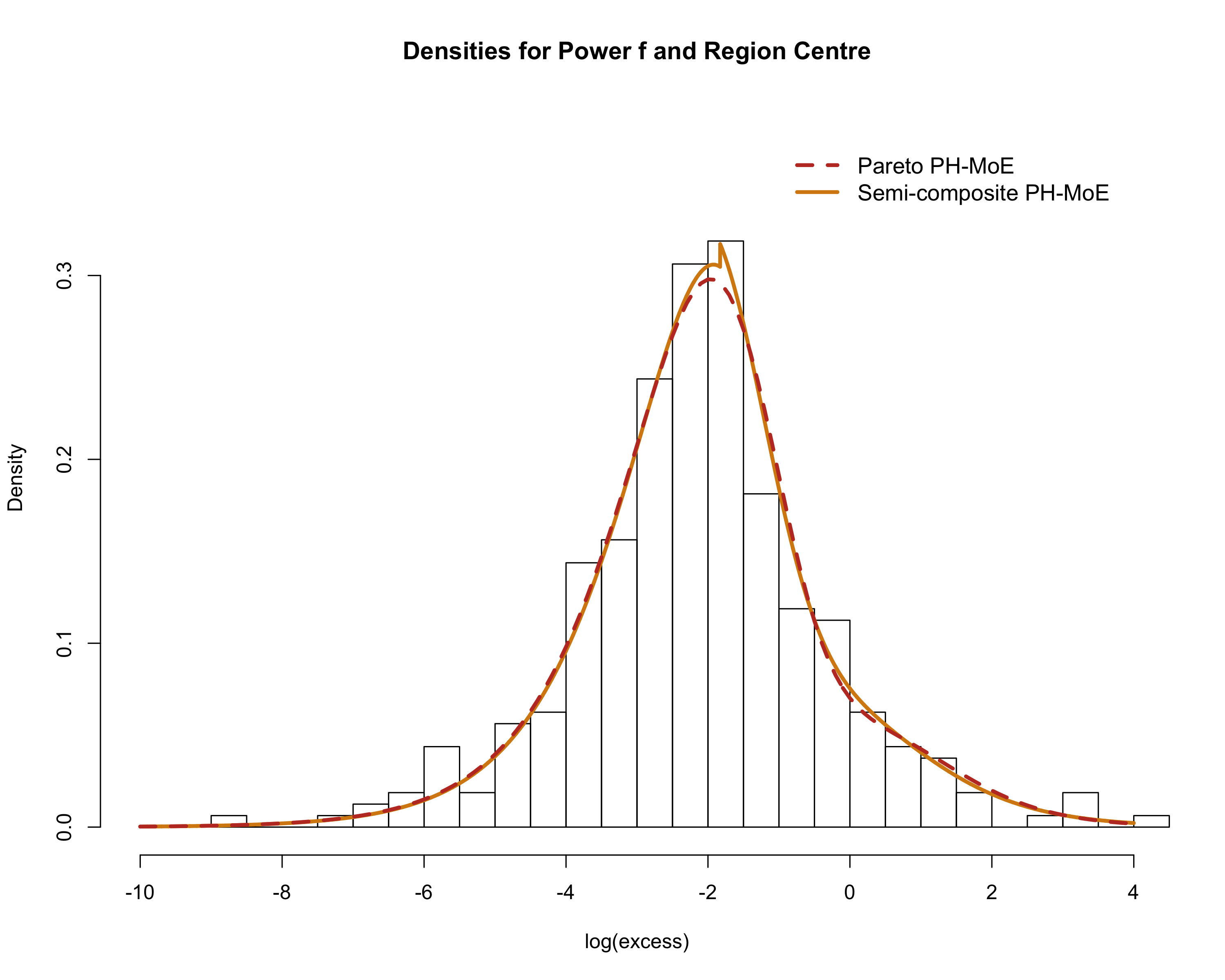}
\includegraphics[width=0.7\textwidth]{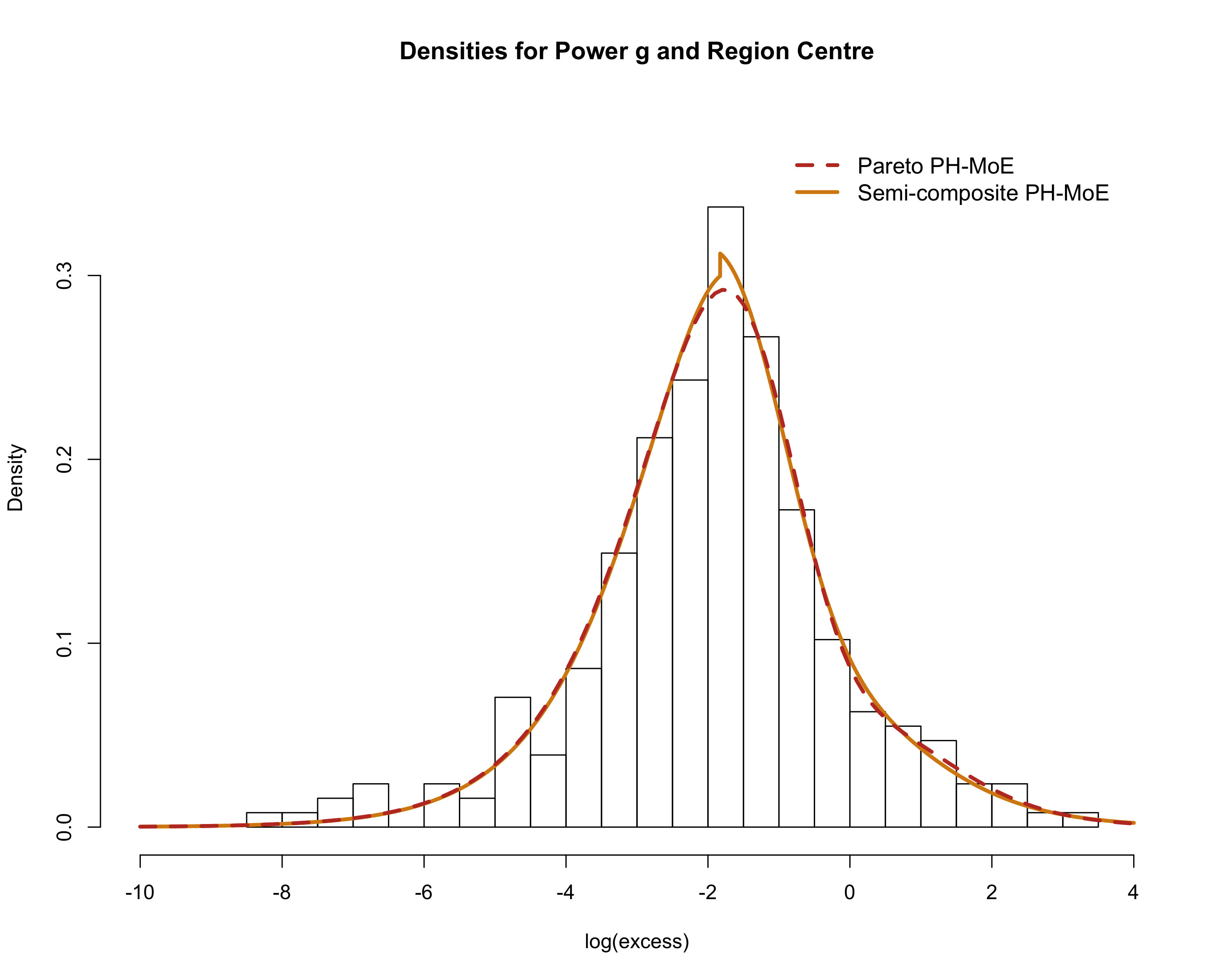}
\caption{Conditional densities for two select covariates.
} \label{two_densities}
\end{figure}

{ \color{black} 
Note that all the five states in both cases communicate, then the resulting tail indices for the PH-MoE models are given by
$$\hat\xi = -1/\max\{\Re\text{Eigen}(\hat\bfT)\} = 0.72,\:0.88,\quad \mbox{respectively},$$}
which, according to the lower panel of Figure \ref{hist_full_data}, are both very reasonable estimates. 
{\color{black}
Further evidence of the quality of the estimation is given in Figure~\ref{residuals}, where we observe that the empirical distribution function of the residuals of PH-MoE models (computed as described in Section~\ref{sec:goodness}) align closely with the distribution function of a standard exponential. This is further supported by applying Kolmogorov-Smirnov tests, for which we obtain a p-value of $0.8717$ for the PH-MoE and $0.8401$ for the semi-composite PH-MoE.
Figure~\ref{intensities} shows how the intensity functions $\lambda$ behave, which may be considered as an infinitesimal ``environment" time change of the underlying phase-type distribution. Another possible extension of our model is to make these transformations dependent on $\bfX$.

\begin{figure}[!htbp]
\centering
\includegraphics[width=0.49\textwidth]{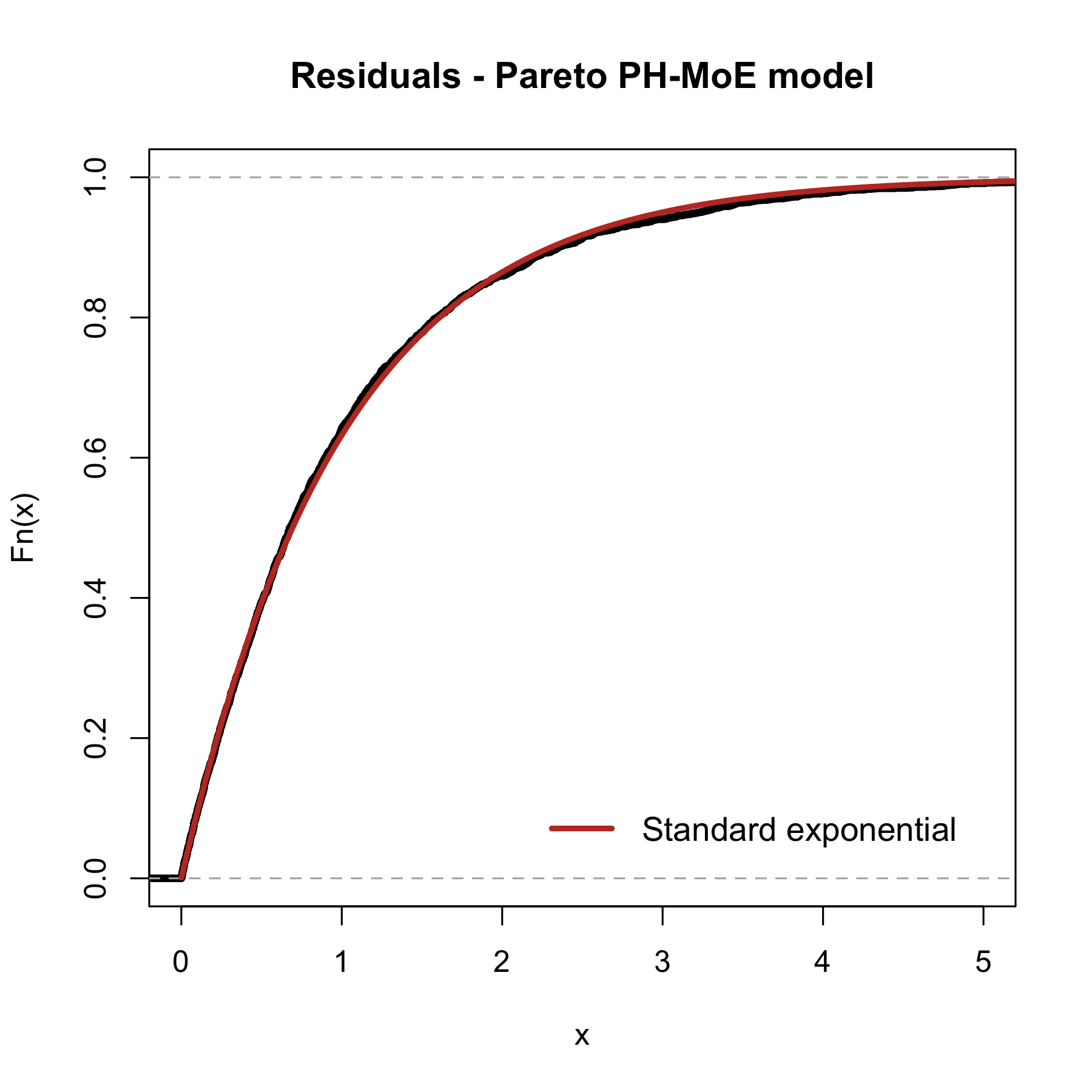}
\includegraphics[width=0.49\textwidth]{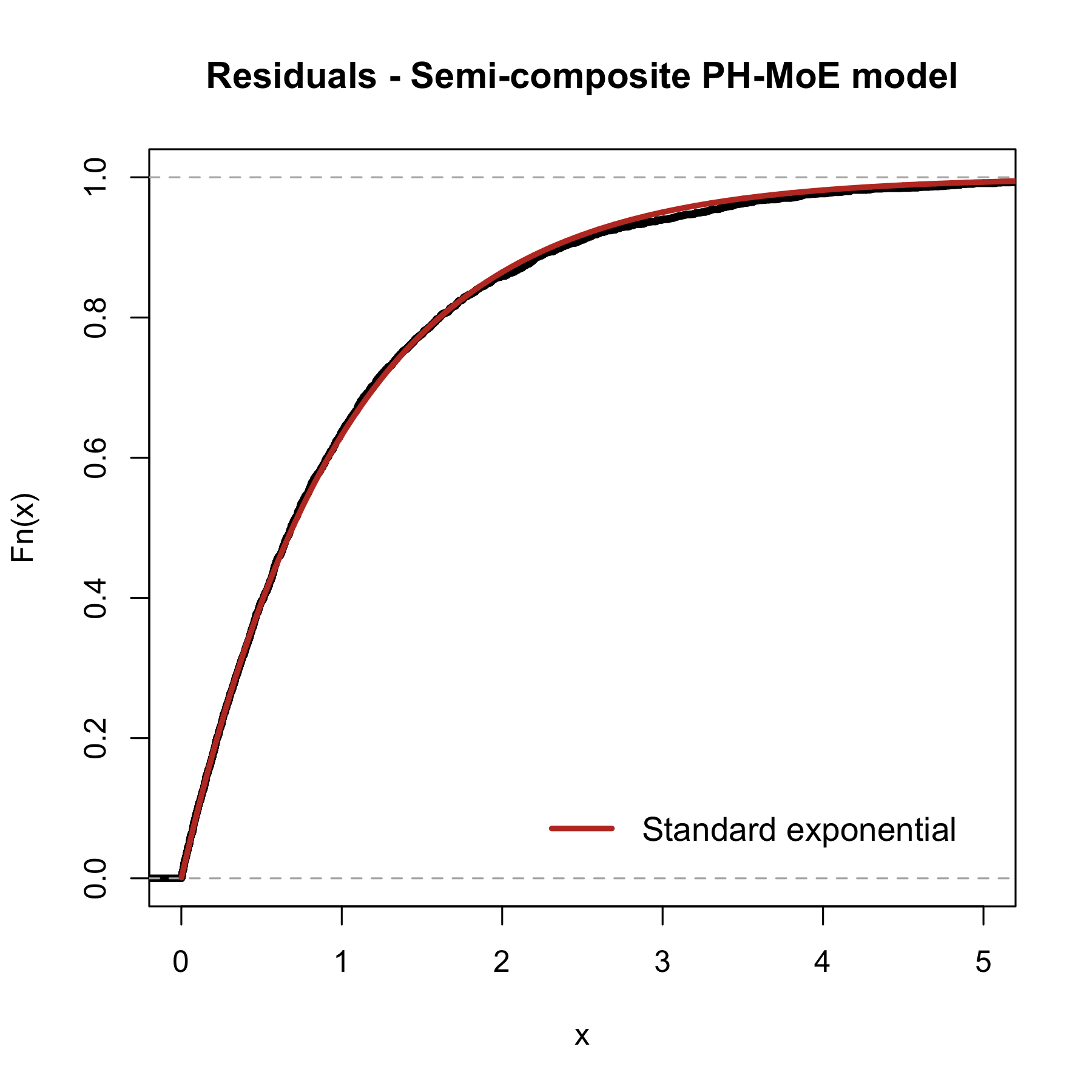}
\caption{CDF for the residuals of the fitted PH-MoE model against CDF of a standard exponential (left panel), and corresponding plot for the semi-composite PH-MoE model (right panel).
} \label{residuals}
\end{figure}

\begin{figure}[!htbp]
\centering
\includegraphics[width=0.7\textwidth]{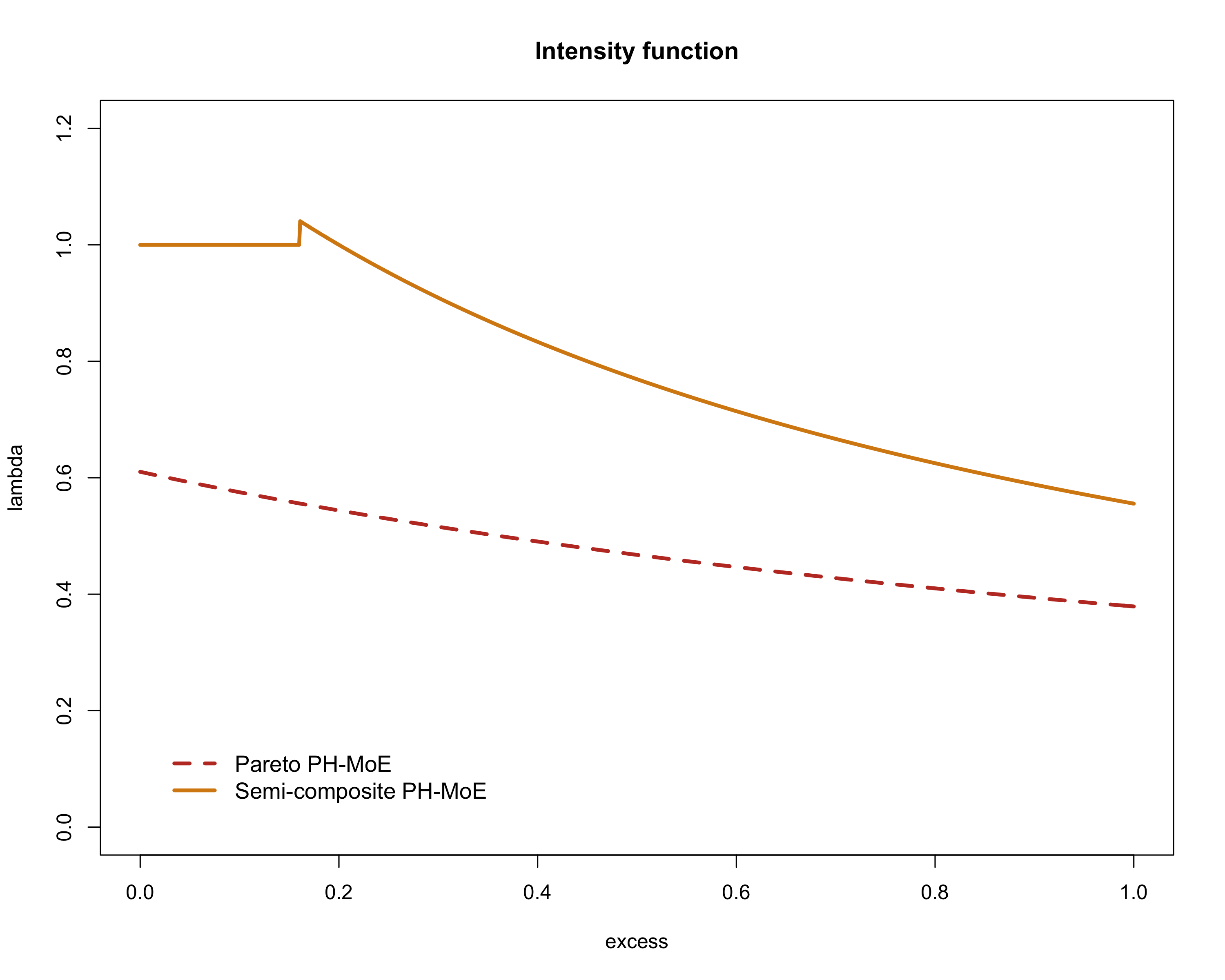}
\caption{Fitted intensity functions for the PH-MoE models.
} \label{intensities}
\end{figure}
}

To extend the analysis to the full data and not only excesses, a direction which is promising is to consider special sub-structures of phase-type distributions. In those cases, the EM algorithm becomes simpler and potentially much faster, consequently enabling the analysis of larger models with more phases being fitted to larger data.

\section{Conclusion}\label{sec:conclusion}
We have presented a claim severities regression model based on PH distributions, incorporating covariates through the initial probability vector, which can be cast into a mixture-of-experts framework. When combined with an inhomogeneity transform, these regression models span distributions with different tail behaviors and possible multimodality. Furthermore, they are flexible and may converge to fairly general regression model specifications. We have derived an effective estimation procedure based on the EM algorithm and a weighted multinomial regression problem and shown its feasibility on synthetic and real insurance data.

Several questions remain open for further research, such as automatic feature selection procedures or using other machine learning methods to predict the initial probability vector. In addition, the analysis of more than one risk, together with their respective claim frequencies, all together in a global model with the same underlying Markov structure is an interesting research direction.

\textbf{Acknowledgement.} MB would like to acknowledge financial support from the Swiss National Science Foundation Project 200021\_191984. JY would like to acknowledge financial support from the Swiss National Science Foundation Project  IZHRZ0\_180549.

\textbf{Declaration} MB and JY declare no conflict of interest related to the current manuscript.

\newpage
\appendix
\section{Estimated coefficients for the PH-MoE models}\label{apA}

{\color{black}

\begin{longtable}{lcc} 
\hline
 State:VariableLevel&  Pareto PH-MoE & Semi-composite PH-MoE \\\hline

2:(Intercept)              & $0.06 \;   (0.51)$          & $0.08 \;   (0.52)$          \\
2:Powere                   & $-0.58 \;   (0.51)$         & $-0.59 \;   (0.51)$         \\
2:Powerf                   & $-33.49$                    & $-25.88$                    \\
2:Powerg                   & $-1.00 \;   (0.52)$         & $-1.03 \;   (0.52)^{*}$     \\
2:Powerh                   & $0.59 \;   (0.46)$          & $0.57 \;   (0.46)$          \\
2:Poweri                   & $-0.24 \;   (0.65)$         & $-0.37 \;   (0.65)$         \\
2:Powerj                   & $40.39 \;   (4.09)^{***}$   & $45.45 \;   (7.05)^{***}$   \\
2:Powerk                   & $0.27 \;   (0.51)$          & $0.23 \;   (0.52)$          \\
2:Powerl                   & $0.95 \;   (0.58)$          & $0.85 \;   (0.59)$          \\
2:Powerm                   & $-17.44 \;   (0.00)^{***}$  & $-17.87$                    \\
2:Powern                   & $-47.86$                    & $-52.91 \;   (0.00)^{***}$  \\
2:Powero                   & $2.01 \;   (1.41)$          & $0.93 \;   (0.90)$          \\
2:RegionBasse-Normandie    & $-11.76 \; (152.10)$        & $-16.73 \;   (0.00)^{***}$  \\
2:RegionBretagne           & $0.21 \;   (0.25)$          & $0.33 \;   (0.25)$          \\
2:RegionCentre             & $-0.76 \;   (0.29)^{**}$    & $-0.55 \;   (0.28)$         \\
2:RegionHaute-Normandie    & $-46.08$                    & $-49.11$                    \\
2:RegionIle-de-France      & $-1.46 \;   (0.39)^{***}$   & $-1.47 \;   (0.39)^{***}$   \\
2:RegionLimousin           & $9.97 \; (243.82)$          & $10.79 \;   (0.00)^{***}$   \\
2:RegionNord-Pas-de-Calais & $13.50 \; (120.65)$         & $24.08 \;  (10.18)^{*}$     \\
2:RegionPays-de-la-Loire   & $-0.22 \;   (0.29)$         & $-0.14 \;   (0.29)$         \\
2:RegionPoitou-Charentes   & $-0.33 \;   (0.34)$         & $-0.22 \;   (0.34)$         \\
3:(Intercept)              & $-94.94 \;  (40.27)^{*}$    & $-96.95 \;   (1.45)^{***}$  \\
3:Powere                   & $38.62 \; (226.78)$         & $62.64 \;   (1.02)^{***}$   \\
3:Powerf                   & $7.73 \;   (4.73)$          & $27.96 \;   (2.15)^{***}$   \\
3:Powerg                   & $57.42 \;  (74.76)$         & $69.91 \;   (0.55)^{***}$   \\
3:Powerh                   & $-2.44 \;  (11.39)$         & $-15.44 \;   (0.00)^{***}$  \\
3:Poweri                   & $78.55 \;  (58.65)$         & $75.11 \;   (0.81)^{***}$   \\
3:Powerj                   & $48.72 \;   (5.22)^{***}$   & $73.89 \;   (4.96)^{***}$   \\
3:Powerk                   & $-71.15 \;   (0.01)^{***}$  & $-98.00 \;   (0.00)^{***}$  \\
3:Powerl                   & $-102.41 \;   (0.98)^{***}$ & $-146.20 \;   (0.00)^{***}$ \\
3:Powerm                   & $156.56 \;   (0.02)^{***}$  & $169.79 \;   (0.15)^{***}$  \\
3:Powern                   & $-7.52 \;   (0.00)^{***}$   & $-7.06 \;   (0.00)^{***}$   \\
3:Powero                   & $9.92 \;   (5.00)^{*}$      & $26.76 \;   (2.24)^{***}$   \\
3:RegionBasse-Normandie    & $11.43 \;   (0.01)^{***}$   & $-12.22 \;   (0.01)^{***}$  \\
3:RegionBretagne           & $-9.48 \;   (0.00)^{***}$   & $-23.21 \;   (0.14)^{***}$  \\
3:RegionCentre             & $39.76 \;  (38.48)$         & $28.88 \;   (1.79)^{***}$   \\
3:RegionHaute-Normandie    & $191.25 \;  (18.38)^{***}$  & $220.06 \;   (0.00)^{***}$  \\
3:RegionIle-de-France      & $88.00 \;  (40.35)^{*}$     & $69.88 \;   (0.82)^{***}$   \\
3:RegionLimousin           & $48.99 \;   (0.00)^{***}$   & $47.03 \;   (0.15)^{***}$   \\
3:RegionNord-Pas-de-Calais & $67.78 \;  (40.04)$         & $61.61 \;   (6.40)^{***}$   \\
3:RegionPays-de-la-Loire   & $17.61 \;  (21.66)$         & $23.22 \;   (1.86)^{***}$   \\
3:RegionPoitou-Charentes   & $37.12 \;  (38.48)$         & $26.96 \;   (1.81)^{***}$   \\
4:(Intercept)              & $0.35 \;   (9.79)$          & $9.79 \;  (36.37)$          \\
4:Powere                   & $-77.73 \;  (18.24)^{***}$  & $-101.22 \;   (0.00)^{***}$ \\
4:Powerf                   & $-33.56 \;   (7.75)^{***}$  & $-34.54 \;  (27.48)$        \\
4:Powerg                   & $-60.72 \;  (41.44)$        & $-74.51 \;  (13.04)^{***}$  \\
4:Powerh                   & $-40.08 \;   (7.75)^{***}$  & $-69.93 \;  (13.04)^{***}$  \\
4:Poweri                   & $-39.41 \;   (7.75)^{***}$  & $-69.09 \;  (13.04)^{***}$  \\
4:Powerj                   & $0.25 \;   (4.04)$          & $-24.40 \;   (6.00)^{***}$  \\
4:Powerk                   & $-106.39 \;   (0.00)^{***}$ & $-115.65 \;   (0.00)^{***}$ \\
4:Powerl                   & $-107.74 \;  (20.20)^{***}$ & $-129.44 \; (227.37)$       \\
4:Powerm                   & $-7.35 \;   (0.02)^{***}$   & $-34.53 \;   (0.15)^{***}$  \\
4:Powern                   & $-79.21 \;  (55.22)$        & $-87.78 \;   (2.60)^{***}$  \\
4:Powero                   & $-117.79$                   & $-114.34 \;   (0.00)^{***}$ \\
4:RegionBasse-Normandie    & $99.85 \;  (23.37)^{***}$   & $107.02 \;  (83.71)$        \\
4:RegionBretagne           & $37.09 \;   (6.50)^{***}$   & $57.67 \;  (25.35)^{*}$     \\
4:RegionCentre             & $60.95 \;  (35.89)$         & $65.00 \;  (25.35)^{*}$     \\
4:RegionHaute-Normandie    & $115.60 \;  (45.65)^{*}$    & $132.22 \; (141.61)$        \\
4:RegionIle-de-France      & $-25.40 \;   (0.00)^{***}$  & $-38.18 \;   (0.00)^{***}$  \\
4:RegionLimousin           & $95.60 \;   (3.29)^{***}$   & $100.51 \;  (16.77)^{***}$  \\
4:RegionNord-Pas-de-Calais & $87.37 \;  (26.09)^{***}$   & $96.12 \;  (20.10)^{***}$   \\
4:RegionPays-de-la-Loire   & $38.32 \;   (6.51)^{***}$   & $58.70 \;  (25.35)^{*}$     \\
4:RegionPoitou-Charentes   & $-0.62 \;   (0.00)^{***}$   & $-1.92 \;   (0.00)^{***}$   \\
5:(Intercept)              & $35.11 \;   (7.75)^{***}$   & $36.13 \;  (27.47)$         \\
5:Powere                   & $-62.31 \;  (41.44)$        & $-75.62 \;  (13.04)^{***}$  \\
5:Powerf                   & $-33.55 \;   (7.75)^{***}$  & $-34.53 \;  (27.48)$        \\
5:Powerg                   & $-62.99 \;  (41.44)$        & $-76.76 \;  (13.04)^{***}$  \\
5:Powerh                   & $-39.53 \;   (7.74)^{***}$  & $-69.38 \;  (13.04)^{***}$  \\
5:Poweri                   & $-40.73 \;   (7.75)^{***}$  & $-70.45 \;  (13.04)^{***}$  \\
5:Powerj                   & $0.34 \;   (4.03)$          & $-24.30 \;   (6.00)^{***}$  \\
5:Powerk                   & $-85.51 \; (149.07)$        & $-102.62 \;   (0.28)^{***}$ \\
5:Powerl                   & $-142.25 \;   (0.00)^{***}$ & $-162.24 \;   (0.08)^{***}$ \\
5:Powerm                   & $-86.81 \;   (0.00)^{***}$  & $-89.54 \;   (0.00)^{***}$  \\
5:Powern                   & $-120.92 \;   (0.00)^{***}$ & $-137.69 \;   (0.00)^{***}$ \\
5:Powero                   & $-121.06 \;   (0.00)^{***}$ & $-120.60 \;   (0.00)^{***}$ \\
5:RegionBasse-Normandie    & $68.24 \;  (19.46)^{***}$   & $83.73 \;  (83.75)$         \\
5:RegionBretagne           & $4.72 \;   (1.65)^{**}$     & $33.51 \;  (18.47)$         \\
5:RegionCentre             & $28.03 \;  (38.04)$         & $40.35 \;  (18.47)^{*}$     \\
5:RegionHaute-Normandie    & $95.90 \;  (22.90)^{***}$   & $119.57 \; (137.82)$        \\
5:RegionIle-de-France      & $-36.69 \;   (7.75)^{***}$  & $-37.44 \;  (27.48)$        \\
5:RegionLimousin           & $64.53 \;   (3.29)^{***}$   & $77.67 \;  (16.79)^{***}$   \\
5:RegionNord-Pas-de-Calais & $54.28 \;  (26.22)^{*}$     & $71.24 \;  (16.20)^{***}$   \\
5:RegionPays-de-la-Loire   & $5.64 \;   (1.68)^{***}$    & $34.32 \;  (18.47)$         \\
5:RegionPoitou-Charentes   & $-0.87 \;   (0.56)$         & $-0.73 \;   (0.58)$         \\\hline
\caption{Significance code: $^{***}p<0.001$; $^{**}p<0.01$; $^{*}p<0.05$. The coefficients associated with state $1$ can be deduced from the constraint $\sum_{k=1}^5 \pi_k(\bfX)=1$.} 
\label{tab:myfirstlongtable}
\end{longtable}

}

{ \color{black}
\section{Inference for phase-type regression models}\label{sec:inference}
Inference and goodness of fit can always be done via parametric bootstrap methods. However, re-fitting a PH regression can be too costly. A first approach is the following general-purpose result:

\begin{theorem}\label{asympt_norm}
Let $\lambda$, $\vect{\eta}:=(\bfalp, \vect{\theta}, \bfT)$ be such that the log-density
$$y\mapsto \log\left[\bfp(\bfalp)\exp\left(\int_0^y\lambda(s;\vect{\theta})ds \bfT\right)\bft\,\lambda(y;\vect{\theta})\right],\quad y>0,$$ satisfies 
Assumptions A0)-A3) of Section 6.3 of \cite{lehmann2006theory} (common supports, identifiable parameters, i.i.d.\ observations, and true parameters in the interior of the parameter space) 
and Assumptions A)-D) of Section 6.5 of \cite{lehmann2006theory} (existence and finite expectation of third derivatives of log-density, strict positive-definiteness of information matrix, and the representation of the latter in terms of expected double partial derivatives of the log-density).

 Then, as the sample size $n\to \infty$, we have that
 \begin{enumerate}
 \item There exist consistent solutions $\hat{\vect{\eta}}_n$ to the likelihood equations.
  \item The following convergence holds:
  $$\sqrt{n}\left(\hat{\vect{\eta}}_n-\vect{\eta}\right)\stackrel{d}{\rightarrow}\mathcal{N}(\vect{0},\mat{\mathcal{I}}^{-1}),$$ where $\mat{\mathcal{I}}$ is the information matrix.
    \item The $j$-th parameter is asymptotically efficient:
  $$\sqrt{n}\left(\hat{\eta}_{jn}-\eta_j\right)\stackrel{d}{\rightarrow}\mathcal{N}(0,[\mat{\mathcal{I}}^{-1}]_{jj}).$$
 \end{enumerate}
\end{theorem}
\begin{proof}
The proof translates directly from Theorem $5.1$ in Section $6.5$ of \cite{lehmann2006theory}.
\end{proof}

Theorem \ref{asympt_norm} is somewhat academic in nature for general PH distributions, because although most of the properties are easy to verify for most parameters and transforms, others are difficult, such as the moment conditions, or near-impossible, such as the identifiability and strict positive-definiteness of the information matrix. Indeed, the latter conditions require that all eigenvalues of $\bfT$ be distinct (although this is not sufficient) and that all parameters be away from the border regions, which is an uncommon scenario when fitting real data.

However, there is still something to be said when it comes to the regression coefficients. Once the EM algorithm has converged, we may consider the parameters $(\vect{\theta},\bfT)$ as nuisance parameters and thus perform inference on the partial likelihood
$\ell(\bfalp|\bfX ,\vect{\theta},\bfT)$, cf. \cite{cox1975partial}, see also \cite{wong1986theory}. As the latter reference suggests, the sub-optimal use of information incurs a loss of efficiency (standard errors should be considered only as lower bounds) which should be weighted against the possible gains in robustness and simplicity of analysis. The usual experiments where the partial likelihood is useful is when the nuisance parameters take values in high-dimensional spaces, making the calculation of the matrix $\mat{\mathcal{I}}$ difficult and not robust. Since the number of nuisance parameters for a PH-MoE model is at least $p^2+1$, it is a clear candidate for benefitting from these tradeoffs.

Another advantage of performing inference on the regression variables alone is that we circumvent making allusion to the possibly non-identifiable parameters of the sub-intensity matrix $\bfT$. In practice, this means that we perform inference during the multinomial regression step (R-step) of Algorithm \ref{alg:IPHMoE}, and use the output to draw conclusions on the statistical significance of the covariates $\bfX$, as well as to perform variable selection.
}

\bibliography{experts.bib}

\end{document}